\documentclass[sts]{imsart}

\usepackage[OT1]{fontenc}
\usepackage[sc]{mathpazo}
\usepackage{pxtx-frak} 

\usepackage{apacite}

\usepackage{xcolor}
\usepackage{marginnote}
\usepackage{color}
\usepackage{float}
\usepackage{ifsym}
\usepackage{mathrsfs}
\usepackage{natbib}
\usepackage{amsthm,amsmath,amstext,amssymb}
\usepackage{stmaryrd}
\usepackage{graphicx}
\usepackage{wasysym}
\RequirePackage[colorlinks,citecolor=blue,urlcolor=blue]{hyperref}
\usepackage{booktabs}

\begin{document}

\numberwithin{equation}{section}
\theoremstyle{plain}
\newtheorem{theorem}{Theorem}[section]
\newtheorem{corollary}[theorem]{Corollary}
\newtheorem{lemma}[theorem]{Lemma}

\theoremstyle{definition}
\newtheorem{definition}{Definition}[section]

\theoremstyle{remark}
\newtheorem{innerexample}{Example}
\newenvironment{example}[1]
  {\renewcommand\theinnerexample{#1}\innerexample}
  {\endinnerexample}
\newtheorem{remark}{Remark}[section]


\newcommand{\up}{\overline{P}}
\newcommand{\lp}{\underline{P}}
\newcommand{\lps}{\underline{P}_{\bullet}}
\newcommand{\ups}{\overline{P}_{\bullet}}



\newcommand{\upb}{\overline{P}_{\mathfrak{B}}}
\newcommand{\lpb}{\underline{P}_{\mathfrak{B}}}
\newcommand{\upd}{\overline{P}_{\mathfrak{D}}}
\newcommand{\lpd}{\underline{P}_{\mathfrak{D}}}
\newcommand{\upg}{\overline{P}_{\mathfrak{G}}}
\newcommand{\lpg}{\underline{P}_{\mathfrak{G}}}
\newcommand{\borel}{\mathscr{B}\left(\Omega\right)}

%

\newcommand{\xr}[1]{\normalmarginpar\marginnote{\scriptsize \bf \color{red}[XL: #1]}}
\newcommand{\red}[1]{{\color{red}{#1}}}
\newcommand{\xc}[1]{\reversemarginpar\marginnote{\scriptsize \bf \color{cyan}[XL: #1]}}
\newcommand{\cyan}[1]{{\color{cyan}{#1}}}

\newcommand{\rr}[1]{\normalmarginpar\marginnote{\scriptsize \bf \color{red}[RG: #1]}}
\newcommand{\rc}[1]{\reversemarginpar\marginnote{\scriptsize \bf \color{cyan}[RG: #1]}}


\begin{frontmatter}
\title{Judicious Judgment Meets Unsettling Updating: \\ Dilation, Sure Loss, and Simpson's Paradox}
\runtitle{Dilation, Sure Loss, and Simpson's Paradox}
\thankstext{T1}{Supported in part by the John Templeton Foundation Grant 52366.}

\begin{aug}
  \author{\fnms{Ruobin} \snm{Gong}\ead[label=e1]{rgong@fas.harvard.edu}}
  \and
  \author{\fnms{Xiao-Li} \snm{Meng}\ead[label=e2]{meng@stat.harvard.edu}}
 
  \runauthor{Gong and Meng}
  \affiliation{Harvard University}
 
  \address[]{1 Oxford Street, Cambridge, MA, 02138 \printead{e1,e2}}
\end{aug}

\begin{abstract}
Statistical learning using imprecise probabilities is gaining more attention because it presents an alternative strategy for reducing irreplicable findings by freeing the user from the task of making up unwarranted high-resolution assumptions. However, model updating as a mathematical operation is inherently exact, hence updating imprecise models requires the user's judgment in choosing among competing updating rules. These rules often lead to incompatible inferences, and can exhibit unsettling phenomena like {\it dilation},  {\it contraction} and {\it sure loss}, which cannot occur with the Bayes rule and precise probabilities. We revisit a number of famous ``paradoxes", including the three prisoners/Monty Hall problem, revealing that a logical fallacy arises from a set of marginally plausible yet jointly incommensurable assumptions when updating the underlying imprecise model. We establish an equivalence between Simpson's paradox and an implicit adoption of a pair of aggregation rules that induce sure loss. We also explore behavioral discrepancies between the \textit{generalized Bayes rule}, \textit{Dempster's rule} and the \textit{Geometric rule} as alternative posterior updating rules for Choquet capacities of order 2. We show that both the generalized Bayes rule and Geometric rule are incapable of updating without prior information regardless of how strong the information in our data is, and that Dempster's rule and the Geometric rule can mathematically contradict each other with respect to dilation and contraction. Our findings show that unsettling updates reflect a collision between the rules' assumptions and the inexactness allowed by the model itself, highlighting the invaluable role of judicious judgment in handling low-resolution information, and the care we must take when applying learning rules to update imprecise probabilities.

\end{abstract}



\end{frontmatter}


\section{There is no free lunch...}\label{sec:intro}

Statistical learning is a process through which models perform updates in light of new information, according to a pre-specified set of operation rules. As new observations arrive, a good statistical model revises and adapts its existing uncertainty quantification to what has just been learned. If a model \emph{a priori} judges the probability of an event $A$ to be $P(A)$, after learning event $B$ happened, it may update the posterior probability according to the Bayes rule: 
\begin{equation*}
		P(A\mid B) = P(A) \cdot \frac{P(B\mid A)}{P(B)}.
\end{equation*}

Exactly one of three things will happen: $P(A\mid B) > P(A)$, $P(A\mid B) < P(A)$, or $P(A\mid B) = P(A)$. Moreover, $P\left(A\mid B\right)>P\left(A\right)$ if and only if $P\left(A\mid B^{c}\right)<P\left(A\right)$, that is, if $B$ expresses positive support for $A$, its complement must express negative support. The comparison of prior and posterior probabilities of $A$ encapsulates its \emph{association} with the observed evidence $B$, a fundamental characterization of the contribution made by a piece of statistical information.

Nevertheless, there exists modeling situations in which associations do not comply with our well-founded intuition. We sketch a series of such examples, well-known from probability classrooms to real life statistical inference, which will serve as the basis of our analysis throughout the paper. Many of them, known as paradoxes, bear multiple solutions that have long been the center of dispute and explication in the literature. What makes all of them thought-provoking is the apparent change from prior to posterior judgments of an event of interest that most will find counterintuitive. That, as we will see, is a consequence of the ambiguity in the probabilistic specification of the model itself, {ambiguity that cannot be meaningfully resolved by any automated rule.}

\subsection{Statistical paradox or imprecise probability?}\label{subsec:intro}

\begin{example}{1}[\textit{Treatment efficacy before and after randomization;} Section \ref{subsubsec:gen_bayes}]\label{example:coin}

Patients Oreta and Tang are participating in a clinical trial, in which one of them will receive treatment I, and the other treatment II, with equal probability. Let $A$ denote the event that Oreta will improve more from this trial than Tang (assuming no ties), and let $B$ denote the event that Tang is assigned to treatment I. Before the treatment is assigned, we clearly have $P\left(A\right)=1/2$ because the situation is fully symmetric ({in the absence of any other information}). However, after the assignment is observed, we seem to have no good idea of the value of either $P\left(A\mid B\right)$ or $P\left(A\mid B^{c}\right)$, other than they are both bounded within $[0, 1]$. 
\end{example}


Example \ref{example:coin} showcases a severe form of ``confusion'' expressed by the model as the prior probability updates to posterior probability in light of {\it any} new information. The precise prior judgments $P\left(A\right)=1/2$ and $P\left(A^c\right)=1/2$ are both bound to suffer a loss of precision by the sheer act of conditioning on {any event in} $\mathcal{B} = \{B, B^c\}$. A central topic of this paper is the {\it dilation} phenomenon, revealed by \cite{good1974little} and investigated in depth by \cite{seidenfeld1993dilation,herron1994extent,herron1997divisive}. A formal definition is given in Section \ref{subsec:dilation}.

\begin{example}{2}[\textit{The boxer, the wrestler, and the coin flip \citep{gelman2006boxer}}; Sections \ref{subsec:dilation} \& \ref{subsec:verdict}]\label{example:boxer}

The greatest boxer and the greatest wrestler are scheduled to fight. Who will defeat the other? Let $Y = 1$ if the boxer wins; $Y = 0$ if the wrestler wins.  Also, let $X = 1$ if a toss of a fair coin yields heads; $X = 0$ if tails. A witness at both the fighting match and the coin flip tells us that $X = Y$. Given this, what is the boxer's chance of winning, $P\left(Y=1\mid X=Y \right)$? 
\end{example}

\begin{example}{3}[\textit{Three prisoners \citep{diaconis1978mathematical, diaconis1983some}}; Sections \ref{subsec:sureloss} \& \ref{subsec:decision}]\label{example:prisoner}
	Three death row inmates $A$, $B$, and $C$ are told, on the night before their execution, that one of them has been chosen at random to receive parole, but it won't be announced until the next morning. Desperately hoping to learn immediately, prisoner $A$ says to the guard: ``Since at least one of $B$ and $C$ will be executed, you'll give away no information about my own chance by giving the name of just one of either $B$ or $C$ who is going to be executed.''  Convinced of this argument, the guard truthfully says, ``$B$ will be executed.'' Given this information, how should $A$ judge his living prospect, $P\left(A\text{ lives}\mid\text{guard says }B\right)$? 
\end{example}


\begin{example}{4}[\textit{Simpson's paradox \citep{simpson1951interpretation,blyth1972simpson};} Section \ref{sec:simpson}]\label{example:simpson}
We would like to evaluate the effectiveness of a novel treatment (experimental) compared to its standard counterpart (control). Let $Z = 1$ denote assignment of the experimental treatment, $0$ the control treatment, and let $Y = 1$ denote the event of a recovery, $0$ otherwise. Let $U \in \{1, 2,...,K\}$ be a covariate of the patients, a $K$-level categorical indicator variable. One could imagine $K$ to be very large, to the extent that the univariate $U$ creates sufficiently individualized strata among the patient population. 

Suppose we learn from reliable clinical studies that the experimental treatment works better than the control for all  $K$ subtypes of patients. That is,
\begin{equation}
p_{k} \equiv P\left(Y=1\mid Z=1,U=k\right) > q_{k} \equiv P\left(Y=1\mid Z=0,U=k\right), \quad k = 1,...,K.  \label{eq:simpson-conditional}
\end{equation}
Nevertheless, field studies consisting of feedback reports from clinics and hospitals seem to suggest otherwise; that on an overall basis, the control treatment cures more patients than the experimental treatment. That is, 
\begin{equation}
\bar{p}_{\text{obs}} \equiv P_{\text{obs}}\left(Y=1\mid Z=1\right) < \bar{q}_{\text{obs}} \equiv P_{\text{obs}}\left(Y=1\mid Z=0\right). \label{eq:simpson-marginal}
\end{equation}
How do we resolve the apparent conflict between the conditional inference in (\ref{eq:simpson-conditional}) and the marginal inference in (\ref{eq:simpson-marginal})?
\end{example}

 


The above examples will be examined in detail in Sections~\ref{sec:updates} through \ref{sec:results}. All of them, despite being disguised with cunning descriptions, share the characteristic of being an {\it imprecise model}. Their narratives imply the existence of a joint distribution, yet only a subset of marginal information is precisely specified.

For instance, in Example~\ref{example:coin}, while the treatment assignment ($B$) is known to be fair prior to randomization, the improvement event $A$ is not measurable with respect to the $B$ margin, effectively posing a Fr\'{e}chet class of joint distributions on the $\left\{A,B\right\}$ space. The only statements we can make about $P(A \mid Z)$ are the trivial bounds $0 \le P\left(A\mid Z\right)\le 1$,
whether $Z=B$ or $Z=B^c$, leading to the dilation phenomenon.  For Example~\ref{example:boxer}, the coin margin $X$ is fully known {\it a priori}, but the relationship between the fighters $Y$ and the coin $X$, crucial for quantifying the event $\{X=Y\}$, is unspecified. In Example~\ref{example:prisoner}, the guard's tendency to report $B$ over $C$ is unspecified in the case that $A$ was granted parole, yet $A$'s survival probability depends critically on this reporting tendency. In all of these examples, the water gets muddied due to an unspecified but necessary piece of relational knowledge, which in turn imposes on the modeler a choice among a multiplicity of updating rules, each supplying a distinct set of assumptions to complement this ambiguity.

\subsection{What do we try to accomplish in this paper?}

Unsettling phenomena to be discussed in this paper reflect unusual ways through which more information can seemingly ``harm'' our existing knowledge of the state of matters. These phenomena are not foreign to statisticians, but are seen as anomalies and paradoxes, far from everyday model building. In fact, whenever there is a fully and precisely specified probability model, none of these phenomena would occur. Wouldn't we all be safer then by staying away from any imprecise model? Quite the contrary, we argue. Imprecise models are unavoidable even in basic statistical modeling, and sometimes they are disguised as precise models only to trick us into blindness. Simpson's paradox, re-examined in Section~\ref{sec:simpson}, is one of such cases. Without acknowledging the imprecise nature of modeling, one is ill-suited to make judicious choices among the updating rules and treatments of evidence.  


We aim to investigate these perceived anomalies as they occur during the updating of imprecise models, and their implications on the choice of updating rules. Imprecise models in statistical modeling are ubiquitous and can be easily induced from precise models through the introduction of external variables. When model imprecision is present, a choice among updating rules is a necessity, and it reflects the modeler's judgment on how statistical evidence at hand should be used. 
With the recent surge of interest in imprecise probability-based and related
statistical frameworks including generalized Fiducial inference \citep{hannig2009generalized}, confidence distribution \citep{hannig2012note, schweder2016confidence} and inferential models \citep{martin2015inferential}, we are compelled to bring attention to the non-negligible choice of combining and conditioning rules for statistical evidence.

The remainder of this paper starts with an introduction to some formal notations of imprecise probabilities in Section \ref{subsec:def}, particularly of Choquet capacities of order 2 as well as belief functions, a versatile special case which can also be formulated as a precise model for imprecise states, that is, set-valued random variables.  Three main updating rules are introduced in Section \ref{subsec:rules}, all of which are applicable to Choquet capacities of order 2. Section \ref{sec:updates} defines dilation, contraction and sure loss as phenomena that happen during imprecise model updating, and Section \ref{sec:simpson} shows how Simpson's paradox is a consequence of an ill-chosen updating rule by establishing its equivalence to a sure loss in aggregation. It also shows how imprecise models can be easily induced from precise ones. Section~\ref{sec:results} compares and contrasts the behavior of the updating rules, especially as they exhibit dilation and sure loss, and illustrates them with an additional example. When do the updating rules differ, and how? We believe these questions will shed light on the means through which information could contribute to imprecise statistical models, a topic we will discuss in Section~\ref{sec:discussion}, among others.



\section{Imprecise probabilities and their updating rules}

This section introduces some formal concepts and notation governing imprecise probability needed within the scope of this paper. Readers who are familiar with such notion may skip to Section \ref{sec:updates}.

\subsection{Coherent lower and upper probabilities}\label{subsec:def}

\begin{definition}[\textit{Coherent lower and upper probabilities}]
Let $\Omega$ be a separable and completely metrizable space
and $\mathscr{B}(\Omega)$ be its Borel $\sigma$-algebra. The {\it lower and upper probabilities} of a set of probability measures $\Pi$ on $\Omega$ are set functions
\begin{equation*}
\lp \left(A\right) = \inf_{P\in\Pi}P\left(A\right) \ {\rm and} \
\up \left(A\right)  = \sup_{P\in\Pi}P\left(A\right), \quad {\rm for\ all}\  A \in \mathscr{B}(\Omega).
\end{equation*}
Respectively, $\lp$ and $\up$ are said to be {\it coherent} lower and upper probabilities if $\Pi$ is closed and convex \citep{walley2000towards}.
\end{definition}
Note that $\lp$ and $\up$ are {\it conjugate} in the sense that $\up\left(A\right) = 1 - \lp\left(A^c\right)$; thus one is sufficient for characterizing the other. We may refer to either $\lp$ or $\up$ individually with the understanding of their one-to-one relationship. Next we introduce Choquet capacities, an important class of imprecise probabilities widely used in robust statistics \citep{huber1973minimax,wasserman1990prior}.

\begin{definition}[\textit{Choquet capacities of order $k$}]
Suppose a coherent lower probability $\lp$ is such that $\{P; P \ge \lp \}$ is weakly compact. $\lp$ is a {\it Choquet capacity of order $k$}, or {\it $k$-monotone} capacity, if for every {Borel-measurable} collection of {$\{A, A_1,...,A_{k}\}$} 
such that $A_{i} \subset A$ for all $i = 1,...,k$, we have
 \begin{equation}\label{eq:choquet}
 	\lp \left(A\right)\ge\sum_{\emptyset\neq I\subset\left\{ 1,...,k\right\} }\left(-1\right)^{\left| I\right|-1}\lp\left(\cap_{i\in I}A_{i}\right)
 \end{equation}	
where $|S|$ denotes the number of elements in the set $S$. Its conjugate capacity function $\up$ is a called a $k$-{\it alternating} capacity, because it satisfies for every {Borel-measurable} collection of $\{A, A_1,...,A_{k}\}$ such that $A \subset A_{i}$ for all $i = 1,...,k$,
  \begin{equation}
 	\up \left(A\right) \le\sum_{\emptyset\neq I\subset\left\{ 1,...,k\right\} }\left(-1\right)^{\left| I\right|-1}\up\left(\cup_{i\in I}A_{i}\right).
\end{equation}
\end{definition}

If a Choquet capacity is $(k+1)$-monotone, it is  $k$-monotone as well: the smaller the $k$, the broader the class. Choquet capacities of order 2 are a special case of coherent lower probability. They satisfy $ \lp \left(A\cup B\right)\ge  \lp  \left(A\right)+ \lp  \left(B\right)-\lp\left(A\cap B\right)$ for all $A,B \in \mathscr{B}\left(\Omega\right)$.
A most special case of Choquet capacities consists of belief functions  \citep{shafer1979allocations}.

\begin{definition}[\textit{Belief function}]
	$\lp$ is called a {\it belief function} if it is a Choquet capacity of order $\infty$, i.e., if (\ref{eq:choquet}) holds for every $k$. 
\end{definition}

It can be easily verified that precise probabilities are a special case of belief functions, and in turn Choquet capacities of any order, thus all of the above are more general constructs than precise probability functions. That being said, they constitute a small class of imprecise probabilities imaginable. Belief functions in particular have their own specializations and limitations when it comes to characterizing imprecise knowledge in probability specifications.  \cite{pearl1990reasoning} noted that belief functions are often incapable of characterizing imprecise probabilities expressed in conditional forms, a category in which Examples~\ref{example:coin} and~\ref{example:simpson} fall, thus neither models are belief function-representable. However, unique to belief function is its intuitive interpretation as a random set object that realizes itself as subsets of $\Omega$. 

\begin{definition}[\textit{Mass function of a belief function}]
If $\lp$ is a belief function, its associated {\it mass function} is the non-negative set function $m:\mathscr{P}\left(\Omega\right)\to\left[0,1\right]$ such that
\begin{equation}\label{eq:mobius}
m\left(A\right)=\sum_{B\subseteq A}\left(-1\right)^{{|A-B|}}\lp \left(B\right),  \quad {\rm for\ all}\ {A \in \borel}
\end{equation}
where $A-B=A\cap B^c$, and the notation $\sum_{B\subseteq A}$ should be understood as taking the sum under the additional constraint that $B\in \borel$ (a convention for the rest of this article, whenever appropriate). Here $m$ satisfies (1) $m\left(\emptyset\right)=0$, (2) $\sum_{B\subseteq\Omega}m\left(B\right)=1$, and (3) $\lp \left(A\right)=\sum_{B\subseteq A}m\left(B\right)$ and is unique to $\lp$.
\end{definition}
Formula (\ref{eq:mobius}) is called the {\it M{\"o}bius transform} of $\lp$ \citep{yager2008classic}. A mass function $m$ induces a precise probability distribution on the subsets of $\Omega$, as the distribution of a random set $\mathcal{R}$.
In Section~\ref{sec:updates} we will invoke the mass function representation of belief functions in Examples~\ref{example:boxer}, \ref{example:prisoner}, and \ref{example:election} (to be discussed in Section~\ref{subsec:poll}).

\subsection{Updating rules for coherent lower {and upper} probabilities} \label{subsec:rules}


To update a set of probabilities $\Pi$ \textit{given} a set $B\in \borel$ is to replace the set function $\underline{P}$ with a version of the \textit{conditional set function} $\lps \left(\cdot \mid B\right)$. The definition of $\lps$ is precisely the job of the updating rule. We emphasize that to say an event is ``given'' does not necessarily mean it is ``observed''. In hypothetical contemplations we often employ conditional statements about all events in a partition, for example $\mathcal{B}=\{B, B^c\}$, even if logically we cannot observe $B$ and $B^c$ simultaneously. Therefore, the phrase ``given'' should be understood as imposing a mathematical constraint \textit{generated} by $B$. When $\Pi$ contains a single, precise statistical model, the Bayes rule entirely dictates how we use the information supplied by $B$. But when $\Pi$ is imprecise and does not possess {\it sharp} knowledge about $B$, i.e., $\lp \left(B\right) < \up \left(B\right)$ \citep{dempster1967upper}, the updating rule itself becomes an imprecise matter. As a consequence, there exists multiple reasonable ways to use the information in $B$, e.g, ``supported by $B$" and ``not contradicted by $B$" generate two different constraints. This raises both flexibility and confusion in defining the updating rules. Here we supply the formal definitions of three viable updating rules for coherent lower and upper probabilities: the {\it generalized Bayes} rule, {\it Dempster's} rule, and the {\it Geometric} rule. Important differences and relationships exist among these rules, as we shall present in Section~\ref{sec:results}.

\subsubsection{Generalized Bayes rule.}\label{subsubsec:gen_bayes}
Recall Example~\ref{example:coin}. Using the notation in \ref{subsec:def}, we rewrite the imprecise model in terms of its prior upper and lower probabilities of event $A$, which are precisely one half:
$\lp(A)=\up(A)=0.5.$
The question is: what are the upper and lower probabilities of $A$ {\it given} the treatment assignments in  $\mathcal{B} = \{B, B^{c}\}$?
For example, a version of the answer supplied here is
\begin{equation*}
\lpb(A \mid B) = 0  \, ,\  \upb(A \mid B) = 1, \quad {\rm and} \quad
\lpb(A \mid B^c) = 0 \, ,\ \upb(A \mid B^c) = 1.
\end{equation*}
The expressions $\lpb$ and $\upb$, where the subscript $\mathfrak{B}$ is for $\mathfrak{B}ayes$, signify the use of the generalized Bayes rule, as defined below. 

\begin{definition}[\textit{Generalized Bayes rule}]\label{def:gen_bayes}
Let $\Pi$ be a closed, convex set of probability measures on $\Omega$.  The conditional lower and upper probabilities according to the {\it generalized $\mathfrak{B}$ayes rule} are set functions $\lpb$ and $\upb$ such that, for $A, B \in \mathscr{B}(\Omega)$,
\begin{eqnarray}
\lpb \left(A\mid B\right) & = & \inf_{P\in\Pi} \frac{P\left(A\cap B\right)}{P\left(B\right)},\label{eq:lpb} \\
\upb \left(A\mid B\right) & = & \sup_{P\in\Pi} \frac{P\left(A\cap B\right)}{P\left(B\right)} \, . \label{eq:upb}
\end{eqnarray}
\end{definition}
That is, the conditional lower and upper probabilities are respectively the minimal and maximal Bayesian conditional probability among elements of $\Pi$. In their definition of the generalized Bayes rule, \cite{seidenfeld1993dilation} worked with the requirement that $\lp \left(B\right) > 0$, which guarantees $P \left(B\right) > 0$  for all $P\in \Pi$; hence the ratios in (\ref{eq:lpb}) and (\ref{eq:upb}) are always well-defined.

The generalized Bayes rule is a most widely employed updating rule for coherent lower and upper probabilities \citep{walley1991statistical}, and is notable for its dilation phenomenon. In Example~\ref{example:coin}, as a consequence of employing the rule, the conclusion appears puzzling: Tang will surely receive one of the two treatments, and one would expect that, in the worst case scenario, learning about the treatment assignment is completely useless, i.e., having no effect on our {\it a priori} assessment of $P(A)$. But how could it be that the knowledge of something can do more harm than being useless?  

To gain a better understanding of the behavior of the generalized Bayes rule, we introduce two alternative updating rules for sets of probabilities as a means of comparison. Both Dempster's rule of conditioning and the Geometric rule were originally proposed for use with the special case of belief functions, however their expressions compose intriguing counterparts to the generalized Bayes rule. Section~\ref{sec:results} is dedicated to a comparison among the trio of rules.

\subsubsection{Dempster's rule.}\label{subsubsec:dempster}
Dempster's rule of conditioning is central to the Dempster-Shafer theory of belief functions \citep{dempster1967upper,shafer1976mathematical}. The conditioning operation is a special case of Dempster's rule of combination, equivalent to combining one belief function with another that puts 100\% mass on one particular subset, $B\in \mathscr{B}(\Omega)$, on which we wish to condition. Specifically, let $\lp$ be a belief function such that $\lp\left(B\right)>0$, and $m$ be its associated mass function given by (\ref{eq:mobius}). Let $\lp_{0}$ be a separate belief function such that its associated mass function $m_{0}\left(B\right)=1$. The conditional belief function $\lpd \left(\cdot\mid B\right)$ is defined as 
\begin{eqnarray*}
\lpd \left(A\mid B\right) & = & \lp \left(A\right)\oplus \lp_{0}\left(B\right),  \quad {\rm for\ all\ } A\in \mathscr{B}\left(\Omega\right),
\end{eqnarray*}
where the combination operator ``$\oplus$'' is defined in \cite{shafer1976mathematical} to imply that the mass function associated with $\lpd \left(\cdot\mid B\right)$ is

\begin{equation}\label{eq:mass-cond-dempster}
m_{\mathfrak{D}}\left(A\mid B\right)=\frac{\sum_{C\cap B=A}m\left(C\right)}{\sum_{C'\cap B\neq\emptyset}m\left(C'\right)}, \quad {\rm for\ all\ }A \in \mathscr{B}\left(\Omega\right). 
\end{equation}

Consequently,  Dempster's rule of conditioning yields the following form.
\begin{definition}[\textit{Dempster's rule of conditioning}] \label{def:dempster}
For $\lp$ a belief function over $\mathscr{B}(\Omega)$ and $\Pi$ the set of probability measures compatible with $\lp$, the lower and upper probabilities according to {\it $\mathfrak{D}$empster's rule of conditioning} are set functions $\lpd$ and $\upd$ such that for $A, B \in \mathscr{B}(\Omega)$ with $\up \left(B\right) > 0$,
\begin{eqnarray}
\lpd \left(A\mid B\right) & = & 1 - \upd \left(A^{c}\mid B\right), \label{eq:lpd} \\
\upd \left(A\mid B\right) & = & \frac{\sup_{P\in\Pi} P\left(A\cap B\right)}{\sup_{P\in\Pi} P\left(B\right)} \, . \label{eq:upd}
\end{eqnarray}
\end{definition}
\noindent
Hence $\upd \left(A\mid B\right)$
differs from $\upb \left(A\mid B\right)$ of (\ref{eq:upb}) by taking the ratio of the suprema, instead of the supremum of the ratio $P\left(A\cap B\right)/P\left(B\right)$. 

An operational view of (\ref{eq:upd}) is helpful for understanding exactly what information is retained by Dempster's rule \citep{gong2017let}. Denote by $\mathcal{R}$ the set-valued random variable representing the distribution as dictated by the mass function corresponding to $\lp$. Dempster's rule of conditioning on set $B$ is akin to taking a ``$B$-shaped cookie cutter'' to all realizations of $\mathcal{R}$, i.e., retaining all intersection sets that $\mathcal{R}$ has with $B$ given it is non-empty, discarding the rest while renormalizing the retained sets such that their mass function $m_{\mathfrak{D}\left(\cdot \mid B\right)}$ as in (\ref{eq:mass-cond-dempster}) normalizes to one.

\subsubsection{The Geometric rule.}
The Geometric rule was proposed by \cite{suppes1977using} as an intended alternative to Dempster's rule. 
\begin{definition}[\textit{The Geometric rule}]  \label{def:geometric}
Let $\lp$ be a belief function as in Definition \ref{def:dempster}. The conditional lower and upper probabilities according to the {\it $\mathfrak{G}$eometric rule} are set functions $\lpg$ and $\upg$ such that for $A, B \in \mathscr{B}(\Omega)$ with $\lp \left(B\right) > 0$,
\begin{eqnarray}
\lpg \left(A\mid B\right) & = & \frac{\inf_{P\in\Pi}P\left(A \cap B\right)}{\inf_{P\in\Pi}P\left(B\right)}, \label{eq:lpg}\\
\upg \left(A\mid B\right) & = & 1 - \lpg \left(A^{c} \mid B\right) \, .\label{eq:upg}
\end{eqnarray}
\end{definition}


Mathematically, the Geometric rule appears to be a natural dual to Dempster's rule, by replacing the latter's suprema for upper probability as in (\ref{eq:upd}) with the infima for lower probability, as in (\ref{eq:lpg}). Operationally, the Geometric rule differs from Dempster's rule by retaining all mass-bearing sets of $\mathcal{R}$ that are contained within $B$, discarding the rest while renormalizing the resulting mass function. Section~\ref{sec:results} further describes some relationships between the two rules. In his review of \cite{shafer1976mathematical}, \cite{diaconis1978mathematical} discussed a paradoxical conclusion for the three prisoners example (reproduced here as Example~\ref{example:prisoner}) using Dempster's rule, and inquired about the option of the Geometric rule as an alternative rule of updating. As we will show in Section~\ref{subsec:sureloss}, the Geometric rule does no better job than Dempster's rule for this paradox, as in fact both exhibit precisely a {\it sure loss} phenomenon.

More updating rules for belief functions exist beyond Dempster's and the Geometric rule, including the disjunctive rule by \cite{smets1993belief} based on set union operations, the open-world conjunctive rule which is the unnormalized version of Dempster's rule as employed in the transferable belief models, as well as others, e.g., \cite{yager1987dempster,kohlas1991reliability,kruse1990specialization}. \cite{smets1991updating} provided a broad overview of an array of updating rules. 

\subsubsection{Applicability to Choquet capacities.}

The generalized Bayes rule was designed to work with coherent sets of probabilities, thus by default applicable to special cases of coherent probabilities such as Choquet capacities of order 2. \cite{wasserman1990bayes} showed that, when applied to prior sets of probabilities that are Choquet capacities of order 2, the posterior sets of probabilities by the generalized Bayes rule remain in the class. Specifically, \cite{fagin1991new} showed that, when they are applied to belief functions, the conditional probabilities will remain a belief function with the expression
\begin{eqnarray}
\lpb \left(A\mid B\right) & = & \lp \left(A\cap B\right)/\left(\lp \left(A\cap B\right)+\up\left(A^{c}\cap B\right)\right), \label{eq:lpb2}\\
\upb \left(A\mid B\right) & = & \up \left(A\cap B\right)/\left(\up \left(A\cap B\right)+\lp\left(A^{c}\cap B\right)\right). \label{eq:upb2}
\end{eqnarray}
The formula of the mass functions induced by the M\"{o}bius transform corresponding to $\lpb \left(\cdot\mid B\right)$ was given in \cite{jaffray1992bayesian}. 

On the other hand, can Dempster's rule and the Geometric rule be applied to sets of probabilities more general than belief functions? Below we show that they can for Choquet capacities of order $k$ where $k \ge 2$, and the corresponding conditional lower probability will remain Choquet capacities of order $k$.

\begin{theorem}
	Let $\lp$ be a $k$-monotone Choquet capacity on $\mathscr{B}(\Omega)$, and event $B$ such that the set functions $\lpd(\cdot \mid B)$ in (\ref{eq:lpd}) and $\lpg(\cdot \mid B)$ in (\ref{eq:lpg}) are well-defined.
	Then, $\lpd(\cdot \mid B)$ and $\lpg(\cdot \mid B)$ are both $k$-monotone.
\end{theorem}

\begin{proof}
To say $\lp$ is $k$-monotone implies for all {Borel-measurable} collections $\{A_1,...,A_k\}$,
\begin{equation*}
	\lp\left(\cup_{i=1}^{k}A_{i}\right) \ge \sum_{i=1}^{k}\lp\left(A_{i}\right)-\sum_{i<j}\lp\left(A_{i}\cap A_{j}\right)+\cdots+\left(-1\right)^{k+1}\lp\left(\cap_{i=1}^{k} A_{i}\right) 
\end{equation*}
or, equivalently, $\up$ is $k$-alternating:
\begin{equation*}
\up\left(\cap_{i=1}^{k}A_{i}\right) \le  \sum_{i=1}^{k}\up\left(A_{i}\right)-\sum_{i<j}\up\left(A_{i}\cup A_{j}\right)+\cdots+\left(-1\right)^{k+1}\up\left(\cup_{i=1}^{k} A_{i}\right).
\end{equation*}

For Dempster's rule, we have
\begin{eqnarray*}
	\upd\left(\cap_{i=1}^{k}A_{i} \mid B \right)	& = &	\frac{\up\left(\left(\cap_{i=1}^{k}A_{i}\right)\cap B\right)}{\up\left(B\right)}
	 = \frac{\up\left(\cap_{i=1}^{k}\left(A_{i}\cap B\right)\right)}{\up\left(B\right)} \\
	&\le &\begin{split}
		\frac{1}{\up\left(B\right)}\cdot\left[\sum_{i=1}^{k}\up\left(A_{i}\cap B\right)-\sum_{i<j}\up\left(\left(A_{i}\cap B \right) \cup \left(A_{j}\cap B\right)\right)+\cdots \right. \end{split} \\ 
	&  & \begin{split} \left. + \left(-1\right)^{k+1}\up\left(\cup_{i=1}^{k}\left(A_{i}\cap B\right)\right)\right]
	\end{split} \\
	& = & \sum_{i=1}^{k}\upd\left(A_{i}\mid B\right)-\sum_{i<j}\upd\left(A_{i}\cup A_{j}\mid B\right)+\cdots \\
    & &	 + \left(-1\right)^{k+1}\upd\left(\cup_{i=1}^{k}A_{i}\mid B\right).
\end{eqnarray*}

Similarly, for the Geometric rule,
  \begin{eqnarray*}
	\lpg\left(\cup_{i=1}^{k}A_{i} \mid B \right)	& = &	\frac{\lp\left(\left(\cup_{i=1}^{k}A_{i}\right)\cap B\right)}{\lp\left(B\right)}
	 = \frac{\lp\left(\cup_{i=1}^{k}\left(A_{i}\cap B\right)\right)}{\lp\left(B\right)} \\
	&\ge & \begin{split} \frac{1}{\lp\left(B\right)}\cdot\left[\sum_{i=1}^{k}\lp\left(A_{i}\cap B\right)-\sum_{i<j}\lp\left(A_{i}\cap A_{j}\cap B\right)+\cdots \right. \end{split} \\
     & & \begin{split} \left. + \left(-1\right)^{k+1}\lp\left(\cap_{i=1}^{k}A_{i}\cap B\right)\right] \end{split} \\
	& = &	\sum_{i=1}^{k}\lpg\left(A_{i}\mid B\right)-\sum_{i<j}\lpg\left(A_{i}\cap A_{j}\mid B\right)+\cdots \\
	& & + \left(-1\right)^{k+1}\lpg\left(\cap_{i=1}^{k}A_{i}\mid B\right).
\end{eqnarray*}
Hence $k$-monotonicity is preserved by both Dempster's and the Geometric rules of updating when applied to $k$-monotone Choquet capacities.
\end{proof}

\section{The unsettling updates in imprecise probabilities}\label{sec:updates}

Because an imprecise model permits, and indeed requires, a choice of updating rule, it may exhibit updates that has troubling interpretations, notably {\it dilation}, {\it contraction} and {\it sure loss}. This section supplies a detailed look at these phenomena. {We emphasize that the subscript ``$\bullet$'' used in the definitions below is crucial because, given the same imprecise model specification, a phenomenon can be induced by one rule but not by another. The choice among updating rules is inseparable from the choice of assumption of a missing information mechanism, and it would be wrong to think that an observable event, as a mathematical constraint, is taken literatim in imprecise probability conditioning. The operational interpretations of Dempster's rule and the Geometric rule presented in Section~\ref{sec:updates} highlight the different uses, by different rules, of the information in the same event being conditioned upon.

\subsection{Dilation and contraction}\label{subsec:dilation}

\begin{definition}[\textit{Dilation}]\label{def:dilation}
Let $A\in \mathscr{B} \left(\Omega\right)$ and $\mathcal{B}$ be a {Borel measurable} partition of $\Omega$. Let $\Pi$ be a closed, convex set of probability measures defined on $\Omega$, $\lp$ its lower probability function, and $\lps$ the conditional lower probability function supplied by the updating rule ``$\bullet$''. We say that $\mathcal{B}$ \emph{strictly dilates} $A$ under the $\bullet$-rule if 
\begin{equation} \label{eq:dilation}
	\sup_{B\in\mathcal{B}}\lps\left(A\mid B\right)<\lp\left(A\right)\le\up\left(A\right)<\inf_{B\in\mathcal{B}}\ups\left(A\mid B\right).
\end{equation}
If either (but not both) outer inequality is allowed to hold with equality, we simply say $\mathcal{B}$ \emph{dilates} $A$ under the said updating rule.
\end{definition} 

Dilation means that the conditional upper and lower probability interval of an event $A$ contains that of the unconditional interval, regardless of which $B$ in the space of possibilities $\mathcal{B}$ is observed. Inference for $A$, as expressed by the imprecise probabilities under the chosen updating rule, will become strictly less precise regardless of what has been learned. This is commonly perceived as unsettling, because one would expect that learning, at least in \emph{some} situations, ought to help the model deliver sharper inference, reflected in a tighter probability interval. But when dilation happens, it seems that as we learn, knowledge does not accumulate and quite the contrary, diminishes surely.

If dilation is something one finds unsettling, the opposing notion, {\it contraction}, should be nothing less. Contraction happens when the posterior upper and lower probability interval becomes strictly contained within that of the prior, regardless of what is being learned.  If a tighter probability interval symbolizes more knowledge, when contraction happens, it is as if some knowledge is created out of thin air.
How could it be that whatever is learned, we could always eliminate a fixed set of values of probability that were {\it a priori} considered possible? If we could have eliminated them by a pure thought experiment that can never fail, why wouldn't we have eliminated them \textit{a priori}? Formally, contraction is defined as follows.

\begin{definition}[\textit{Contraction}]
Let $A$, $\mathcal{B}$ and $\lps$ be the same as in Definition \ref{def:dilation}. We say that $\mathcal{B}$ \emph{strictly contracts} $A$ under the $\bullet$-rule if 
\begin{equation} \label{eq:contraction}
\lp \left(A\right)<\inf_{B\in\mathcal{B}}\lps \left(A\mid B\right)\le\sup_{B\in\mathcal{B}}\ups\left(A\mid B\right)<\up\left(A\right).	
\end{equation}
If either (but not both) outer inequality is allowed to hold with equality, we simply say $\mathcal{B}$ \emph{contracts} $A$ under the said updating rule.

\end{definition}


We now illustrate these two unsettling updating phenomena using Example \ref{example:boxer}, although we defer the discussion of their interpretations to Section~\ref{sec:discussion}.

\begin{example}{\ref{example:boxer} cont.}[\textit{The boxer, the wrestler, and the coin flip}]

By the setup of the model, we know precisely that the coin is fair:
\begin{equation}\label{eq:boxer-prior-x}
	P(X = 0) = P(X = 1) = 1/2.
\end{equation}
However, no information is available about either fighter's chance of winning. That is, if we assume the probability of a boxer's win $P(Y = 1) = \pi$, $\pi$ is allowed to vary between $[0, 1]$. Then according to the imprecise model,
\begin{equation}\label{eq:boxer-post-y}
	\lp(Y=1)=0 \, , \qquad \up(Y=1)=1
\end{equation}
and similarly so for the wrestler's win: $\lp(Y=0)=0, \up(Y=0)=1$. The known probabilistic margins specify a belief function, as displayed in Table~\ref{table:boxer}.

\begin{table}[H]
\caption{\label{table:boxer} Example~\ref{example:boxer} (boxer and wrestler): mass function representation of the belief function model}
\def\arraystretch{1.5}
\begin{tabular}{|c|c|c|}
\hline 
 & coin lands head, either fighter wins  & coin lands tails, either fighter wins \tabularnewline
 & $\left(X,Y\right)\in\left\{ 1\right\} \times\left\{ 0,1\right\} $ &  $\left(X,Y\right)\in\left\{ 0\right\} \times\left\{ 0,1\right\} $\tabularnewline
\hline 
$m\left(\cdot\right)$ & 0.5 & 0.5\tabularnewline
\hline
\end{tabular}
\end{table}

When told $X=Y$, how should the model at hand be revised?  Two aspects are worth noting: 

\smallskip
\subparagraph{i) Posterior inference for the fighters.} 
As \cite{gelman2006boxer} noted, Dempster's rule contracts the boxer's chance of winning, because
\begin{eqnarray*}
	\lpd \left(Y=1\mid X=Y\right) = 1/2,
& \qquad & \upd\left(Y=1\mid X=Y\right) = 1/2, \\
	\lpd \left(Y=1\mid X\neq Y\right) = 1/2,
& \qquad & \upd\left(Y=1\mid X\neq Y \right) = 1/2
\end{eqnarray*}
which are strictly contained within the vacuous prior probability interval as in (\ref{eq:boxer-post-y}). The calculations given the two alternative conditions $X = Y$ and $X \neq Y$ are identical due to symmetry of the setup. In contrast, the generalized Bayes rule cannot contract  vacuous prior interval, in this example and in general (see Theorem~\ref{thm:vacuous-interval}). 

\smallskip
\subparagraph{ii) Posterior inference for the coin.}
Intriguingly, the generalized Bayes rule dilates the precise {\it a priori} information (\ref{eq:boxer-prior-x}) on the coin's chance of coming up heads, because
\begin{eqnarray*}
	\lpb \left(X=1\mid X=Y\right) = 0,
& \qquad & \upb\left(X=1\mid X=Y\right) = 1, \\
	\lpb \left(X=1\mid X\neq Y\right) = 0,
& \qquad & \upb\left(X=1\mid X\neq Y\right) = 1.
\end{eqnarray*}
In contrast, Dempster's intervals remain identical to that of the prior interval under either $X=Y$ or $X\neq Y$. Notice that in this example, $\lp\left(X=Y\right)=\lp\left(X\neq Y\right)=0$ hence the Geometric rule is not applicable. The generalized Bayes rule in the sense of \citeauthor{seidenfeld1993dilation} (see Definition \ref{def:gen_bayes}) is not applicable either; however since the the model is a belief function, we use \citeauthor{fagin1991new}'s results as given in (\ref{eq:lpb2}) and (\ref{eq:upb2}) to obtain the above expressions. This is equivalent to minimizing and maximizing over the restricted sets of probabilities $\{P: P\ge \lp\ \& \ P\left(X=Y\right) > 0 \}$ and $\{P: P\ge \lp\ \& \ P\left(X \neq Y\right) > 0 \}$ respectively, thus avoiding ill-defined probability ratios. 

\end{example}
 

\subsection{Sure loss}\label{subsec:sureloss}

The next type of updating anomaly is even more unsettling, as it is usually regarded as an infringement on the logical coherence of probabilistic reasoning. 

\begin{definition}[\textit{Sure loss}]\label{def:sure-loss}
Let $A$, $\mathcal{B}$, $\lp$, and $\lps$ be the same as in Definition \ref{def:dilation}. We say that $\mathcal{B}$ incurs \emph{sure loss} in $A$ under the $\bullet$-rule if either
\begin{equation}
	\inf_{B\in\mathcal{B}} \lps \left(A\mid B\right)>\up\left(A\right), \label{eq:sure-loss}
\end{equation}
or
\begin{equation}
\sup_{B\in\mathcal{B}}\ups\left(A\mid B\right)<\lp\left(A\right). \label{eq:sure-gain}
\end{equation}
\end{definition}

Sure loss describes a universal and uni-directional displacement of probability judgment before and after conditioning on any event from a subalgebra. That is, after learning anything, the event in question becomes altogether more (or less) likely than before. The terminology ``sure loss'' stems from the Bayesian decision-theoretic context, where probabilities are seen to profess personal preferences contingent on which one is willing to make bets. If $\mathcal{B}$ incurs sure loss in $A$, the beholder of $\lp$ and $\lps$ as her personal prior and posterior imprecise probabilities respectively, can be made to commit a compound bet with a guaranteed negative payoff. 

To see this, let $s,t$ be two numbers such that $\inf_{B\in\mathcal{B}}\lps \left(A\mid B\right)>s>t>\up \left(A\right)$, that is, we assume sure loss in the form of (\ref{eq:sure-loss}). Since $t>\overline{P}\left(A\right)$, I shall accept a bet for which I pay $1-t$, get $1$ back if $A$ did not occur, and nothing if it did, because my expected payoff is $P(A^c)-(1-t)=t-P(A)\ge t-\up(A)>0$. On the other hand, since $\lps \left(A\mid B\right)>s$ for all $B$, contingent on any $B$ I shall also accept bets for which I pay $s$, get $1$ back if $A$ did occur and nothing if it did not, because regardless of which $B$ occurs, my expected payoff $P(A\mid B)-s\ge \inf_{B\in\mathcal{B}}\lps\left(A\mid B\right)-s>0$. It therefore seems perfectly logical for me to take both bets, as both are expected to have positive return. However, if I do take both bets, then the compound bet is the one with guaranteed payoff of only $1$, less than what I have paid for $1-t+s$ because $s>t$. Therefore, endorsing $\lps$ as the updating rule means one is willing to accept a finite collection of bets and be certain to lose money, a trademark incoherent behavior.

Note that if $\mathcal{B}$ incurs sure loss in $A$ in the form of (\ref{eq:sure-loss}), it equivalently incurs sure loss in $A^{c}$ as well in the form of (\ref{eq:sure-gain}), though perhaps the term {\it sure gain} would be more appropriate -- in \'{E}mile Borel's words, the former the ``imbecile'' and the latter the ``thief''. Whenever a distinction is necessary, we will use the term sure gain in addition to sure loss to highlight the directionality of displacements of posterior probability intervals compared to that of the prior, and will otherwise follow the pessimistic convention (which seems to be a hallmark of statistical or probabilistic terms, such as ``risk", ``regret'', ``regression'')  of the literature and use ``sure loss'' to refer to both situations if non-ambiguous.

We emphasize again that both dilation and sure loss, as concepts describing the change from prior to posterior sets of probabilities, are contingent upon the updating rule. Even with the same imprecise probability model $\lp$, the same partition $\mathcal{B}$ and the same event $A$, it can well be the case that $\mathcal{B}$ dilates $A$ under one rule and induces sure loss in $A$ under the other. Example~\ref{example:prisoner} below is a situation in which all three rules behave very differently, and Section~\ref{sec:results} is dedicated to a characterization of their differential behavior.

We are now ready to take a careful look at the three prisoners paradox.

\begin{example}{\ref{example:prisoner} cont.}[\textit{Three prisoners}]

What do we have about the probabilistic model behind the three prisoners'? Since exactly one of the three prisoners will receive parole randomly, the prior probabilities of living for each of them are all exact:
$$P\left(A\text{ lives}\right)=P\left(B\text{ lives}\right)=P\left(C\text{ lives}\right)=1/3.$$
Furthermore, since the guard cannot lie, he has no choice on who to report if the inquirer $A$ does not receive parole. That is,
$$P(\text{guard says }C\mid B\text{ lives})=P(\text{guard says }B\mid C\text{ lives})=1.$$
The above probability specification can be expressed as a belief function model, with mass distribution dictated by the known model margins as represented in Table~\ref{table:prisoner}:
	
\begin{table}[H]
\caption{\label{table:prisoner} Example~\ref{example:prisoner} (three prisoners): mass function representation of the belief function model}	
\def\arraystretch{1.5}
\begin{tabular}{|c|c|c|c|}
\hline  
 & A lives, guard says \{B, C\} & B lives, guard says C & C lives, guard says B\tabularnewline
\hline 
$m\left(\cdot\right)$ & $1/3$ & $1/3$ & $1/3$\tabularnewline
\hline 
\end{tabular}
\end{table}

We see from the specification that what remains unknown is, in case $A$ indeed receives parole, the propensity of the guard reporting either $B$ or $C$ as dead had he the freedom to choose: 
\begin{equation}\label{eq:delta_b}
	{\delta_B} = P(\text{guard says }{ B}\mid { A} \text{ lives}) \in [0,1].
\end{equation}
As a consequence, the posterior probability of $A$ living is
\begin{equation}\label{eq:delta}
P({ A}  \text{ lives} \mid \text{guard says } { B}) = \delta_B/ (1+\delta_B).
\end{equation}
This extra degree of freedom $\delta_B$ fully characterizes the set of probabilities implied by the model.

There is a long literature documenting the variety of modes of reasoning to this problem, e.g., \cite{mosteller1965fifty} and \cite{morgan1991let} which invoked a similar construction as the $\delta_B$ above, in explicating the reasons why many are seemingly intuitive yet riddled with logical fallacies. Four types of ``popular'' answers are reproduced below, reflecting different ways of treating the unknown value $\delta_B$. What's interesting is that, as we will see, three of these answers correspond to those given by the three conditioning rules respectively.

\smallskip
\subparagraph{i) The indifferentist: assumption of ignorability}
One of the most commonly made assumptions is that the guard has no preference one way or the other about who to report when given the freedom, that is, $\delta_B =1/2$, thus
\begin{equation*}
P({ A}  \text{ lives} \mid \text{guard says } { B},\delta_B =1/2) = {1}/{3}.
\end{equation*}
That is to say, prisoner $A$ would not have benefitted from the knowledge that $B$ is going to be executed, precisely as he claimed to the guard to begin with. The assumption of guard's indifference is equivalent to the {\it ignorability} assumption commonly employed in the treatment of missing and coarse data. Despite being intuitive, the assumption is not backed by the model description per se. Neither the posited imprecise model nor the data as reported by the guard can supply any logical evidence to support the ignorability assumption. Therefore, the assertion that ignorability is ``intuitive'' is a judgment that can be as unreasonable as any other seemingly less intuitive ones, such as the ones below.

\smallskip
\subparagraph{ii) The optimist: Dempster's rule}
Applying Dempster's rule, we have
\begin{equation*}
\lpd (A\text{ lives}\mid\text{guard says }B) = 1/2 \, , \qquad \upd (A\text{ lives}\mid \text{guard says }B) = 1/2. 
\end{equation*}
Thus prisoner $A$ felt happier now that his chance of survival increased from $1/3$ to $1/2$. This happiness is gained from assuming the optimistic scenario of $\delta_B=1$, that is, the guard chose a reporting mechanism that has the highest likelihood given $A$ lives. However, one realizes that the guard could have only reported either $B$ or $C$, both fully symmetrical in the prior. Had the guard said $C$ would be executed, $A$ would again apply Dempster's rule, thus grow happier following the same logic by effectively assuming $\delta_C= P(\text{guard says }{C}\mid {A} \text{ lives})=1$. Under the assumption that the guard cannot lie and cannot refuse to answer, $\delta_B+\delta_C=1$, thus $\delta_B$ and $\delta_C$ cannot be $1$ simultaneously. Hence the reasoning that whatever the guard says, the probability of $A$ living will go up from $1/3$ to $1/2$, which is equivalent to assuming the impossible $\delta_B=\delta_C=1$, is a direct consequence of a logical fallacy.  

\smallskip
\subparagraph{iii) The pessimist: the Geometric rule}
Applying the Geometric rule, we have
\begin{equation*}
\lpg (A\text{ lives}\mid\text{guard says }B) = \upg (A\text{ lives}\mid \text{guard says }B) = 0
\end{equation*}
and, by symmetry,
\begin{equation*}
\lpg (A\text{ lives}\mid\text{guard says }C) = \upg (A\text{ lives}\mid \text{guard says }C) = 0. 
\end{equation*}
This answer is perhaps the most striking among all, directly pointing at the absurdity of the assumptions behind the updating rule within this context. Upon hearing anything, prisoner $A$ will deny himself of any hope of living, effectively assuming $\delta_B=0$ if guard says $B$ and $\delta_C=0$ if guard says $C$, two assumptions that are incommensurable with each other because $\delta_B+\delta_C=1$, much in the same way as the previous case with Dempster's rule.

\smallskip 
\subparagraph{iv) The conservatist: generalized Bayes rule}
The solution suggested by \cite{diaconis1978mathematical}, and indeed supplied by the generalized Bayes rule, is
\begin{equation}\label{eq:prisoner-a-given-b}
		\lpb (A\text{ lives}\mid \text{guard says }B) = 0 \, , \qquad \upb (A\text{ lives}\mid \text{guard says }B) = 1/2.
\end{equation}
This answer is a direct consequence of (\ref{eq:delta}). As $\delta_B$ varies within $[0,1]$ without any further assumption, one is bound to concur with (\ref{eq:prisoner-a-given-b}). The caveat to it, however, is that again due to prior symmetry of $B$ and $C$, the generalized Bayes rule will also yield 
\begin{equation*}
\lpb (A\text{ lives}\mid\text{guard says }C) = 0 \, , \qquad \upb (A\text{ lives}\mid \text{guard says }C) = 1/2.
\end{equation*}
Hence, the generalized Bayes rule results in posterior probability intervals strictly containing the prior probability in all situations.
\end{example}

Our use of the vocabulary ``optimism'', ``pessimism'' and ``conservatism'' to refer to the three updating rules is informed by the interpretation of their respective posterior inference under the effective assumptions they each impose, and is reminiscent of that of \cite{fygenson2008modeling} for modeling of extrapolated probabilities. These ideological differences illuminate the dynamics among the updating rules for imprecise probability, and highlight the pedagogical significance of the three prisoners' paradox itself. In this example, Dempster's rule updates its conditional lower probability to be greater than that of its prior upper probability thus incurs sure loss of the form (\ref{eq:sure-loss}), the Geometric rule behaves the opposite way and incurs sure loss of the form (\ref{eq:sure-gain}), and the generalized Bayesian rule exhibits dilation. As far as unsettling updating goes, there seems to be no escape regardless of which rule to choose. How on earth then do we draw a conclusion?

\begin{figure}
\noindent \begin{centering}
\includegraphics[width=.7\textwidth]{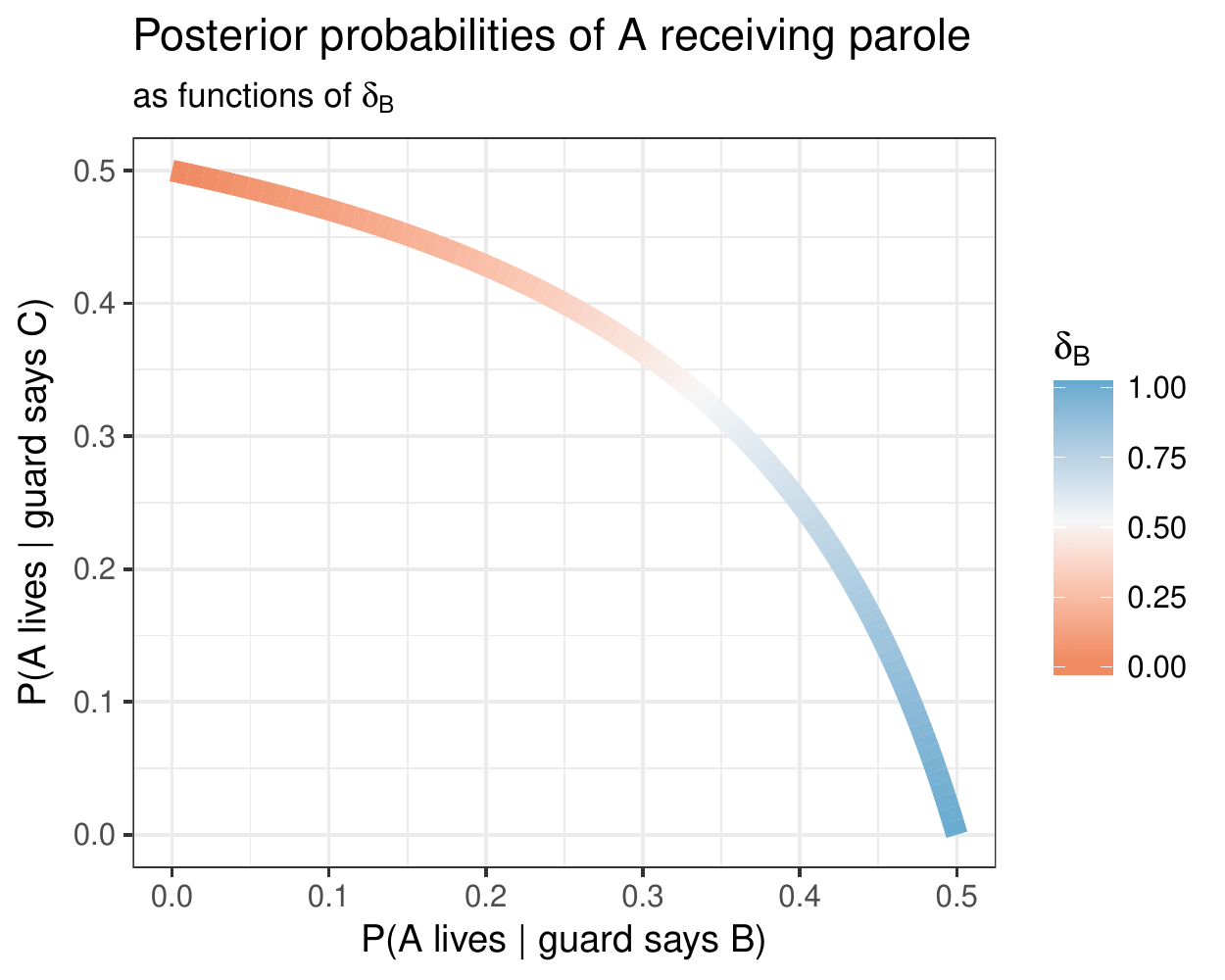}
\par\end{centering}
\caption{\label{fig:prisoner} Posterior probabilities of prisoner $A$ receiving parole given the guard's two possible answers, as a function of the guard's reporting bias $\delta_B$ (\ref{eq:delta_b}). }
\end{figure}

Reading through the literature, the dilated answer supplied by the generalized Bayes rule is the most accepted solution to the paradox. As counterintuitive as it may be, dilation is a professed consequence of an overfitting nature of the generalized Bayes rule, for the rule is inclusive of all possibilities allowed within the ambiguous model, to the point of simultaneously admitting assumptions that are {\it incommensurable} with one another. As we saw previously, the upper conditional probability $\upb (A\text{ lives}\mid \text{guard says }Z) = 1/2$ is achieved under the assumption $\delta_Z=1$, where $Z=C, B$. Similarly, the lower conditional probability $\lpb (A\text{ lives}\mid\text{guard says }Z) = 0$ is achieved when $\delta_Z=0$, with $Z=C, B$. Since $\delta_C+\delta_B=1$,
$\delta_C$ and $\delta_B$ cannot simultaneously be $0$ or $1$: indeed, when one is $1$ the other must be $0$. Hence the permissible value of the \textit{pair} 
\begin{equation*}
\{x=P(A\text{ lives}\mid \text{guard says }B), y=P(A\text{ lives}\mid \text{guard says }C) \}	
\end{equation*}
forms a one-dimensional curve  $y = (1-2x)/(2-3x)$ inside the square $[0, 1/2]\times[0, 1/2]$, as depicted in Figure \ref{fig:prisoner}. For a given conditioning event $Z$, the generalized Bayes rule achieves its extremes by seeking a distribution that itself depends on $Z$, namely, a \textit{conditioning-dependent} conditional distribution $P^{(Z)}(\cdot|Z)$, a clear case of overfitting. Understanding the hidden {\it incommensurability} is important for preventing logical fallacies such as reasoning under the (wrong) assumption that $\{x, y\}$ can take any value inside the square $[0, 1/2]\times[0, 1/2]$. We will return to the three prisoners again in Section~\ref{subsec:decision} to discuss its inferential implications. In particular, the three prisoners' paradox is a direct variant of the Monty Hall problem, which possesses a clean, indisputable decision recommendation.

\subsection{What's unsettled in unsettling updates?}\label{subsec:unmeasurable}

In case some readers are not yet completely put off by the unsettling updates, we would like to offer a few words about when, as well as when not, one {\it should} find dilation or sure loss unsettling. It seems to us that the attitude one should take towards these phenomena is contingent upon the way the underlying probability model is interpreted. 

Specifically, dilation is troubling when the set of probabilities is used as a description of uncertain inference. For example, if the probability interval is regarded as an approximation to some underlying true probability state, akin to a confidence or posterior interval to an estimand, knowing that the interval will surely grow wider in the posterior is indeed counterproductive since the goal of inference in most cases is to tighten the interval. But in this sense, the sure loss phenomenon may just be fine, since it is common to derive disjoint yet equally valid confidence or posterior intervals from the same sampling posterior distribution, without violating any classic rules of probabilistic calculation. 

On the other hand, as explained in Section~\ref{subsec:sureloss}, the lower and upper probabilities can be taken as acceptable prices of a gamble. Under this interpretation, any strategy that induces sure loss is absolutely unacceptable. However in this case, dilation is much less to be worried about: a strictly wider interval in the posterior will simply exclude the player from engaging in the conditional bet, and does not violate coherence in a decision-theoretic sense.



With precise probabilities, by conditioning on an observable event we are simply imposing a restriction to the sub-space that is defined by that event, which itself is assumed to be measurable with respect to the original probability space. With imprecise probabilities, not all events are measurable with respect to the imprecise probability model specified on the full joint space, and a crucial way the updating rules differ from one another is how they make use of this supplied conditioning information. Therefore, for any of the updating rules to function at all, they must build within themselves a particular ``mechanism'' of imposing the mathematical restriction specified by the observable event, when it is not currently measurable with respect to the set of probabilities the rule aims to update, much in the same way as a sampling mechanism \citep{kish1965survey} or missing-data mechanism \citep{rubin1976inference}. The fact that dilation and sure loss cannot happen under the precise probability does not necessarily render them undesirable: the quality of this inference hinges on the quality of the final action they recommend. Bringing these anomalies to light allows us to study their implications, especially those unfamiliar or unexpected, on the final action.


\section{Simpson's paradox: an imprecise model with aggregation sure loss}\label{sec:simpson}

One may well think that the examples discussed in Section \ref{sec:intro} lie on the boundary, if not outside, of the realm of mainstream statistical modeling. Imprecise models do not seem to be the kind of thing one just stumbles upon, they exist only when one makes them exist. We argue that such is not the case, that all precise models are really just the tip of an ``imprecise model iceberg''. That is, every precise model is a fully specified margin nested within a larger, ever-augmentable model, with extended features we have not allowed entrance to the modeling scene, and possibly lacking the knowledge to do so precisely. 

Here is a concrete way to induce an imprecise model from a precise one. Take any precise model with dimensions $\left(X_{1},...,X_{p}\right)$ which jointly merit a known multivariate distribution. If the model is expanded to include a previously unobservable margin $X_{p+1}$, all of a sudden the state space becomes $\left(p+1\right)$-dimensional. The resulting augmented model becomes imprecise, for as many as $2^{p}$ new marginal relationships are left to be specified or learned -- that is, between $X_{p+1}$ and any subset of $\left(X_{1},...,X_{p}\right)$. In the regression setting where a multivariate Normal model is assumed for the previous $p$ variables, one seemingly straightforward way is to go and model $\left(X_{1},...,X_{p},X_{p+1}\right)$ as jointly Normal, which is already a very strong assumption that takes care of a vast majority of the $2^{p}$ relationships. Even under such drastic simplification, the mean and $p+1$ bivariate covariances are still left to specify, resulting in a family of $\left(p+1\right)$-dimensional Normal models.

In reality, the relationship between the existing $\left(X_{1},...,X_{p}\right)$ and a new $X_{p+1}$ is often something the analyst is neither knowledgeable nor comfortable assuming. This is the case when $X_{p+1}$ is a lurking variable in observational studies which may have strong collinearity with subsets of  $\left(X_{1},...,X_{p}\right)$. Using the language of imprecise probability, we now turn to decipher Simpson's paradox, a famous and familiar setting with its far-reaching significance. The very occurrence of Simpson's paradox is proof that we have employed, likely due to lack of control, an aggregation rule that has incurred sure loss in inference.

\begin{example}{\ref{example:simpson} cont.}[\textit{Simpson's paradox}]
Following the setup in Section \ref{sec:intro}, Simpson's paradox refers to an apparent contradiction between an inference on treatment efficacy at an aggregated level,  $\bar{p}_{\text{obs}} < \bar{q}_{\text{obs}}$, and the inference at the disaggregated level when the covariate type of the patient has been accounted for: $p_{k} > q_{k}$ for all $k = 1,...,K$. Indeed, how can a treatment be superior than its alternative in every possible way, yet be inferior \emph{overall}?
\end{example}

\subsection{Explicating the aggregation rules underlying the Simpson's paradox}
Denote for $k=1,...,K$,
\begin{equation*}
	u_{k} = P\left(U=k\mid Z=1\right), \qquad v_{k} = P\left(U=k\mid Z=0\right).
\end{equation*}
Here, ${\bf u}$ and ${\bf v}$ reflect the demographic distribution of the populations receiving the experimental and control treatments, respectively. By the law of total probability,
\begin{equation}\label{eq:simp-p}
	\bar{p} = {\bf p}^\top{\bf u} \quad {\rm and}\quad \bar{q} = {\bf q}^\top{\bf v}, \end{equation}
and thus given fixed {\bf p} and {\bf q}, $\bar{p}$ and $\bar{q}$ are functions of ${\bf u}$ and ${\bf v}$ respectively. The marginal probabilities $\bar{p}$ and $\bar{q}$ are meant to describe an event under conditions of inferential interest, in this case, patient recovery within the two treatment arms. We refer to ${\bf u}$ and ${\bf v}$ as \emph{aggregation rules}, functions that map conditional probabilities to a marginal probability, which is in reverse direction compared to \emph{updating rules} as discussed in the previous sections, which are maps from a marginal probability to a set of conditional probabilities.


Typically, measurements between different conditions are made for the purpose of a comparison, such as the evaluation of an causal effect of treatment $Z$ on outcome $Y$. A comparison between $\bar{p}$ and $\bar{q}$ is {\it fair} if and only if the aggregation rules they employ are identical, that is, ${\bf u} = {\bf v} $ as in (\ref{eq:simp-p}). This is what it means to say the comparison has been made between apples and apples. Such is the case if no confounding exists between the covariate $U$ and the propensity of assignment, i.e., $U \perp Z$. 

Clearly, when ${\bf u} = {\bf v} $, $\bar{p} > \bar{q}$ if $p_k>q_k$ for all $k$. Hence Simpson's paradox is mathematically impossible within a fair comparison. However, for a given observed pair  $\bar{p}_{\text{obs}}$ and $\bar{q}_{\text{obs}}$, have we been careful enough to enforce the \emph{de-facto} aggregation rules to equal the ideal one? That is, do we have that the observed comparison is {\it fair enough}, i.e., there exists a ${\bf v}$ such that
\begin{equation} \label{eq:fair}
{\bf u}_{\text{obs}}	\doteq	{\bf v} \qquad \text{and} \qquad
{\bf v}_{\text{obs}}	\doteq	{\bf v}	?
\end{equation}


For certain values of ${\bf p}$ and ${\bf q}$, it is entirely possible that suitable realizations of $\left({\bf u}_{\text{obs}}, {\bf v}_{\text{obs}}\right)$ within the $K$-simplex could result in $\bar{p}_{\text{obs}} < \bar{q}_{\text{obs}}$. To be exact, these are ${\bf p}$ and ${\bf q}$ values satisfying {$\max_{k}q_{k}>\min_{k}p_{k}$}. At least one, and possibly both realizations of ${\bf u}_{\text{obs}}$ and ${\bf v}_{\text{obs}}$  play differentially to the relative weaknesses of ${\bf p}$, i.e., coordinates of smaller magnitude, and the strengths of ${\bf q}$ accordingly. When this preferential weighting, also known as \emph{confounding}, is strong enough to reverse the perceived stochastic dominance of the outcome variable under either treatment, an apparent paradox is induced. Randomization procedures effectively put quality guarantees on the fairness of comparison; as the sample size $n$ grows larger, (\ref{eq:fair}) holds with high probability with deviations quantifiable with respect to ${\bf p}$ and ${\bf q}$ that is immune against all $U$, observed or unobserved.

\begin{figure}
\noindent \begin{centering}
\includegraphics[width=.47\textwidth]{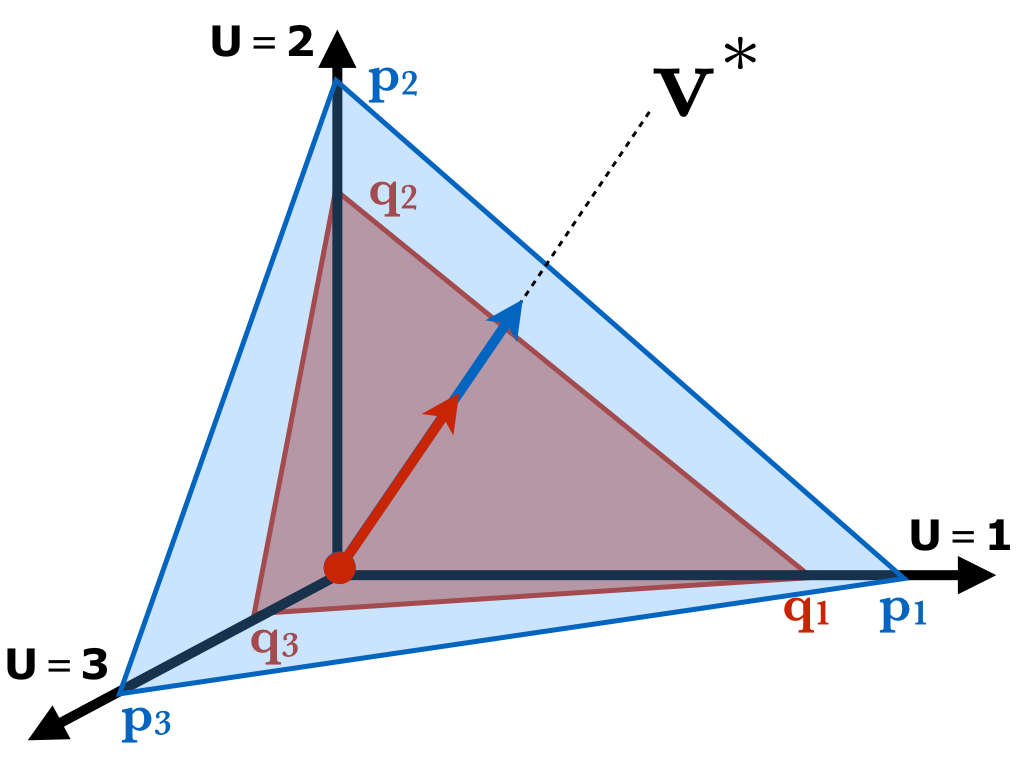}\hfill
\includegraphics[width=.47\textwidth]{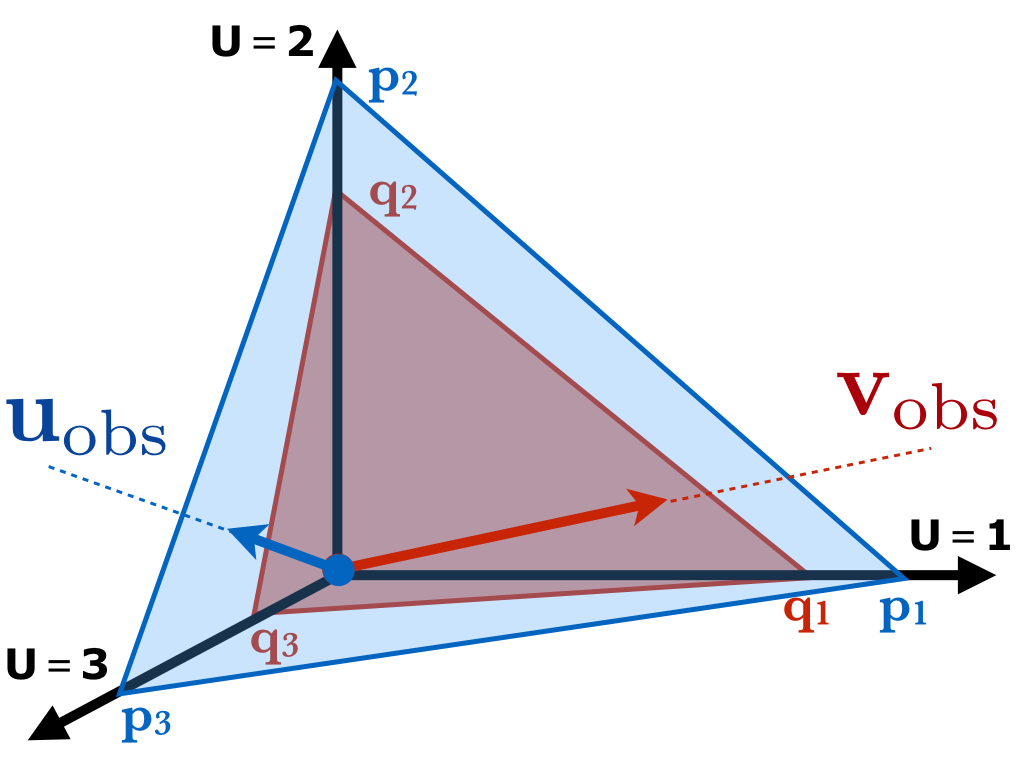}
\par\end{centering}
\caption{\label{fig:simpson} Ideal aggregating rules guarantee the comparison between treatment arms is made on a fair ground. Observed Simpson's paradox is  strong evidence that the {\it de-facto} aggregating rules are fair for comparison. Left: if $p_{k} > q_{k}$ for all $k$, then ${\bf p}^{\top}{\bf v} > {\bf q}^{\top}{\bf v}$ for all ${\bf v}$; Right: disparate ${\bf u}_{\text{obs}}$ and ${\bf v}_{\text{obs}}$ make possible $p_{\text{obs}} < q_{\text{obs}}$. Note that $\Pi$ in Theorem~\ref{thm:simpson-sure-loss} is the convex hull sandwiched between the blue (${\bf p}$) and red (${\bf q}$) hyperplanes in the first octant.}
\end{figure}

\subsection{The paradox is sure loss}
Simpson's paradox is reminiscent of the ``sure loss'' phenomenon we saw in earlier sections. Indeed, when not conditioned on $U$, if asked to pick a bet between the experimental and control treatments, we would prefer the control treatment over the experimental one.
But once conditioned on $U$, the experimental treatment suddenly became the superior bet regardless of $U$'s value. 
One is thus set to surely lose money by engaging in a combination of these two bets. That is,  We formalize this idea in the following theorem, where $\mathcal{S}_K$ is the standard $K$-simplex defined by
$\{(v_1,\ldots, v_K):\   \sum_{k=1}^Kv_k=1;\   v_k\ge 0,  \ k=1,\ldots, K\}.$

\begin{theorem}[\textit{Equivalence of Simpson's paradox and aggregation sure loss}]\label{thm:simpson-sure-loss}

Let $\Lambda$ be a convex hull in $[0,1]^{K}$ characterized by the pair of element-wise upper and lower bounds $\left({\bf p},{\bf q}\right)$. That is,
\[
\Lambda=\left\{ \boldsymbol{\lambda}\in\left[0,1\right]^{K}:q_{k}\le\lambda_{k}\le p_{k}, \  
 k=1,...,K\right\}.
\]
Let $\mathcal{V}\subseteq \mathcal{S}_K$ be a closed set (with respect to the Euclidean distance on $\mathbb{R}^K$) of \emph{desirable aggregation rules}, and let ${\bf u}\in \mathcal{S}_K$. Then, ${\bf u}$ incurs sure loss on $\Lambda$ relative to $\mathcal{V}$ if and only if $\left({\bf u},{\bf v}\right)$ induces Simpson's paradox in $\left({\bf p},{\bf q}\right)$ for all ${\bf v}\in\mathcal{V}$.
\end{theorem}

\begin{proof}
Denote the set of marginal probability derived from $\Lambda$ under the desirable aggregating rule $\mathcal{V}$ as $\mathcal{P}_{\mathcal{V}}=\left\{ \boldsymbol{\lambda}^{\top}{\bf v}:\boldsymbol{\lambda}\in\Lambda,{\bf v}\in\mathcal{V}\right\}. $ By the closeness of both $\Lambda$ and $\mathcal{V}$ (with respect to the Euclidean distance on $R^K$), we have 
\begin{equation}\label{eq:pv}
\inf\mathcal{P}_{\mathcal{V}}
	=\inf_{{\bf v}\in\mathcal{V}}{\bf q}^{\top}{\bf v}\quad  {\rm and}\quad 
    \sup\mathcal{P}_{\mathcal{V}}
	=	\sup_{{\bf v}\in\mathcal{V}}{\bf p}^{\top}{\bf v},
\end{equation}
and 
\begin{equation}\label{eq:pu}
{\bf p}^{\top}{\bf u}=\sup_{\boldsymbol{\lambda}\in\Lambda}\boldsymbol{\lambda}^{\top}{\bf u}	
\quad {\rm and} \quad
{\bf q}^{\top}{\bf u}=\inf_{\boldsymbol{\lambda}\in\Lambda}\boldsymbol{\lambda}^{\top}{\bf u}.
\end{equation} 
Employing Definition \ref{def:sure-loss}, to say that ${\bf u}$ incurs sure loss on $\Lambda$ relative to  $\mathcal{V}$ means that 
\begin{equation}\label{eq:sure}
\sup_{\boldsymbol{\lambda}\in\Lambda}\boldsymbol{\lambda}^{\top}{\bf u}<\inf\mathcal{P}_{\mathcal{V}}
\quad {\rm or}
\quad 
\inf_{\boldsymbol{\lambda}\in\Lambda}\boldsymbol{\lambda}^{\top}{\bf u}>
\sup\mathcal{P}_{\mathcal{V}}.
\end{equation}
On the other hand, to say that for every ${\bf v}\in \mathcal{V}$, $\left({\bf u},{\bf v}\right)$ induces Simpson's paradox in $\left({\bf p},{\bf q}\right)$ means that 
\begin{equation}\label{eq:simp}
{\bf p}^{\top}{\bf u}<\inf_{{\bf v}\in\mathcal{V}}{\bf q}^{\top}{\bf v} \quad {\rm or} \quad {\bf q}^{\top}{\bf u}>\sup_{{\bf v}\in\mathcal{V}}{\bf p}^{\top}{\bf v}.
\end{equation}
By (\ref{eq:pv})-(\ref{eq:pu}), conditions (\ref{eq:sure}) and (\ref{eq:simp}) are trivially the same, and hence the theorem.
\end{proof}

We remark that, in Definition \ref{def:sure-loss}, sure loss is defined with respect to a single conditioning rule because the prior/marginal lower and upper probabilities $\lp$ and $\up$ are treated as given. Such is not the case with the sure loss concept in Theorem \ref{thm:simpson-sure-loss}. This is why we must first define a set of  desirable rules ${\bf v} \in {\mathcal{V}}$ which implies a prior/marginal probability interval, before discussing the behavior of the other aggregation rule ${\bf u}$ relative to it.  One can check that the relationship between ${\bf u}$ and ${\bf v}$ is reciprocal, that is, if ${\bf u}$ induces sure loss relative to ${\bf v}$, then  ${\bf v}$ induces sure loss relative to ${\bf u}$. Thus, we can talk about an \emph{aggregation scheme} as an ordered pair of  rules $\left({\bf u},{\bf v}\right) \in K\text{-simplex}^{2}$, and its characteristics as whether it incurs sure loss with respect to itself,  whether it induces the paradox in $\left({\bf p},{\bf q}\right)$, and so on. 

Indeed, this is the case with respect to the \emph{atomic lower and upper probability} (ALUP) model of \cite{herron1997divisive}. A set of probabilities $\Pi_{\left({\bf p},{\bf q}\right)}$ is an ALUP generated by $\left({\bf p},{\bf q}\right)$, where ${\bf p},{\bf q} \in [0,1]^{K}$, if
\begin{equation}\label{eq:alup}
\Pi_{\left({\bf p},{\bf q}\right)}=\left\{ \boldsymbol{\pi}\in {\mathcal S}_K:\sup\pi_{k}=p_{k}, \ \inf\pi_{k}=q_{k}\right\} 
\end{equation}
A connection between ALUP model and Simpson's paradox is made below.

\begin{lemma}[\textit{ALUP models}]
If an aggregation scheme $\left({\bf u},{\bf v}\right)$ induces Simpson's paradox in $\left({\bf p}, {\bf q}\right)$,  it incurs sure loss with respect to itself on the ALUP model $\Pi_{\left({\bf p},{\bf q}\right)}$, as defined in (\ref{eq:alup}).
\end{lemma}

\begin{proof}
Without loss of generality, suppose an aggregation scheme $\left({\bf u},{\bf v}\right)$ induces Simpson's paradox in $\left({\bf p},{\bf q}\right)$ in the form of ${\bf p}^{\top}{\bf u}=\sup_{\boldsymbol{\lambda}\in\Lambda}\boldsymbol{\lambda}^{\top}{\bf u}<\inf_{\boldsymbol{\lambda}\in\Lambda}\boldsymbol{\lambda}^{\top}{\bf v}=\boldsymbol{q}^{\top}{\bf v}$.
But since $\Pi_{\left({\bf p},{\bf q}\right)}$ is a closed and convex subset of $\Lambda$, we have $\sup_{\boldsymbol{\lambda}\in\Lambda}\boldsymbol{\lambda}^{\top}{\bf u}\ge\sup_{\boldsymbol{\pi}\in \Pi_{(p, q)}}\boldsymbol\pi^{\top}{\bf u}$ and $\inf_{\boldsymbol{\lambda}\in\Lambda}\boldsymbol{\lambda}^{\top}{\bf v}\le\inf_{\boldsymbol{\pi}\in\Pi_{(p, q)}}\boldsymbol{\pi}^{\top}{\bf v}$, hence the ``only if'' part of Theorem \ref{thm:simpson-sure-loss} still holds.
\end{proof}

\subsection{Implication on inference}

In Example \ref{example:simpson}, the description of the model is precise with the conditional values ${\bf p}$ and ${\bf q}$, as well as the marginal values $\bar{p}_{\text{obs}}$ and $\bar{q}_{\text{obs}}$. The model is imprecise, and in fact completely vacuous, on the aggregation rules $\left({\bf u}_{\text{obs}}, {\bf v}_{\text{obs}}\right)$ which gave rise to the observed marginal values. 

In order for the observed marginal probabilities $\bar{p}_{\text{obs}}$ and $\bar{q}_{\text{obs}}$ to yield a meaningful comparison, we must have clear answers to the following two questions regarding ${\bf u}_{\text{obs}}$ and ${\bf v}_{\text{obs}}$:
\begin{enumerate}
	\item Are they equal?
	\item What is the mutual value ${\bf v}$ they both should be equal to?
\end{enumerate}
An affirmative answer to the first question ensures that $\bar{p}_{\text{obs}}$ and $\bar{q}_{\text{obs}}$ are at least on a comparable footing. For example, for the evaluation of an causal effect of $Z$ on $Y$, regardless of the population of interest, it must be ensured that no confounding between the covariate $U$ and the propensity of assignment took place, i.e., $U \perp Z$. That is why Simpson's paradox is a sanity check for any apparent causal relationship, as the paradox constitutes sufficient (but not necessary) evidence there is non-negligible confounding between $U$ and $Z$, a telltale sign that one is comparing apples with oranges. Much classic and contemporary literature on causal inference sensitivity analysis, e.g., \cite{cornfield1959smoking, ding2016sensitivity}, hinge on establishing deterministic bounds to exclude scenarios that are in essence Simpson's paradoxes, as well as quantifying the probability of population-level paradox given observed paradox in the sample, e.g., \cite{pavlides2009likely}. If the assignment $Z$ cannot be controlled in one or both treatment arms, the aggregation rule is no longer chosen by the investigator/statistician but rather left self-selected, in all or in part by the observational mechanism. In particular, if arbitrary confounding can be present in both treatment arms, ${\bf u}$ and ${\bf v}$ can take up any value in the $K$-simplex. It is also entirely possible that controlled randomization or weighting is available in only one of the treatment arms, or on a subset of levels of $U$, reflecting an aggregation rule as a mixture of intentional choice and self-selection. 

It is also crucial that the ideal aggregation rule ${\bf v}$, the mutual value for both ${\bf u}_{\text{obs}}$ and ${\bf v}_{\text{obs}}$, is a conscious choice made by the investigator to reflect the scientific question of interest. Two typical situations that give rise to natural choices of ${\bf v}$ are:
\begin{itemize}
	\item to make inference about population average treatment effect, choose ${\bf v}$ to be the oracle probability distribution of patients' covariates in the population;
	\item to make inference about a particular patient's treatment effect, choose
	\[ {\bf v} = \left(\begin{array}{ccccccc}
0 & \cdots & 0 & 1_{\left(U_{i}=k\right)} & 0 & \cdots & 0\end{array}\right)'\]
the indicator vector matching the patient's covariate value $U_i$ with its level $k$.
\end{itemize} 
One can devise a range of choices of ${\bf v}$ to reflect any amount of intermediate pooling within what is deemed as the relevant subpopulation. As discussed in \cite{liu2014comment,liu2016there}, what defines the game of \emph{individualized inference} is picking the ${\bf v}$ at the appropriate resolution level while subject to the tradeoff between population relevance and estimation robustness.

Choosing the right ${\bf v}$ and enforcing ${\bf u}_{\text{obs}} = {\bf v}_{\text{obs}} = {\bf v}$ is not merely a mathematical decision to be made on paper, but rather entails action in a real-life observational environment, one that likely involves the physical activities of stratification and randomization such as controlled experiments and survey designs. Only through doing so can we make sure the de-facto aggregation rules are equal to the ideal rule, or equivalently that we know executable ways to adjust for the differences between these quantities should there be any, e.g., through retrospective weighting.  Failure to acknowledge the distinction and potential differences among ${\bf v}$, ${\bf u}_{\text{obs}}$, and ${\bf v}_{\text{obs}}$ paves the way not only for Simpson's paradox, but also equivalently for endorsing mythical statistical aggregation rules with the potential to exhibit incoherent behavior, and the worst of all, to mislead ourselves in making the wrong (treatment) decisions, a sure loss in real sense.

\section{Behavior of updating rules: some characterizations}\label{sec:results}

This section presents some results on the behavior of the three updating rules discussed in this paper. We begin with the intuitive ones and progress towards those that are perhaps surprising.  Unless otherwise noted,  this section assumes that $\lp$ is a Choquet capacity of order 2 on $\mathscr{B}\left(\Omega\right)$, and $\Pi = \{P: P \ge \lp \}$, the closed and convex set of probabilities compatible with it. Recall $\lpb$, $\lpd$, and $\lpg$ are the conditional lower probability functions according to the generalized Bayes (Def. \ref{def:gen_bayes}), Dempster's (Def. \ref{def:dempster}) and the Geometric rules (Def. \ref{def:geometric}) respectively.

\subsection{Generalized Bayes rule cannot contract nor induce sure loss}

\begin{lemma}\label{lemma:inf-of-inf}
	Let $\mathcal{B} = \{B_1, B_2, ...\}$ be a measurable and denumerable partition of $\Omega$. For any  $A \in \mathscr{B}\left(\Omega\right)$, we have
	\begin{eqnarray*}
		\inf_{Z\in\mathcal{B}}\lpb\left(A\mid B_{i}\right)\le\lp\left(A\right), \quad {\rm and} \quad
		\sup_{B_{i}\in\mathcal{B}}\upb\left(A\mid B_{i}\right)\ge\up\left(A\right).
	\end{eqnarray*}
\end{lemma}

\begin{proof}
	The proof is by contradiction. Assume that $\inf_{B_{i}\in\mathcal{B}}\lpb\left(A\mid B_{i}\right) > \lp\left(A\right)$. For the given $A$, because $\Pi$ is a closed set, there exists a $P^{(A)} \in \Pi$ such that $P^{(A)}\left(A\right) = \lp \left(A\right)$. Note that the superscript in $P^{(A)}$ reminds us that this probability measure can vary with the choice of $A$.  But this does not affect the validity of applying the total probability law under this chosen $P^{(A)}$, which leads to
	 \begin{eqnarray*}
	 	\lp\left(A\right)=P^{(A)}(A)& = &	\sum_{i=1}^{\infty}P^{(A)}\left(A\mid B_{i}\right)P^{(A)}\left(B_{i}\right) \\
	 &\ge & 	\sum_{i=1}^{\infty}\lpb\left(A\mid B_{i}\right)P^{(A)}\left(B_{i}\right) 
	 \ge 	\sum_{i=1}^{\infty}\inf_{B_{i}}\lpb\left(A\mid B_{i}\right)P^{(A)}\left(B_{i}\right)  \\
	& > & 	\sum_{i=1}^{\infty}\lp\left(A\right)P^{(A)}\left(B_{i}\right) 
	 = 	\lp\left(A\right),
	 \end{eqnarray*}
resulting in a contradiction. The same argument applies for the upper probability of $A$, that if $\sup_{B_{i}\in\mathcal{B}}\upb\left(A\mid B_{i}\right) < \up\left(A\right)$, then using
$\up \left(A\right)=\tilde P^{(A)}\left(A\right)$,
\begin{equation*}		\up\left(A\right)\le\sum_{i=1}^{\infty}\upb\left(A\mid B_{i}\right)\tilde P^{(A)}\left(B_{i}\right)<\sum_{i=1}^{\infty}\up\left(A\right)\tilde P^{(A)}\left(B_{i}\right)=\up\left(A\right),
	\end{equation*}
    and hence again a contradiction.
    \end{proof}
	

A direct consequence of Lemma \ref{lemma:inf-of-inf} is the following thorem.
\begin{theorem}\label{thm:genbayes-coherence}
Let $\mathcal{B}$ be a denumerable and measurable partition of $\Omega$, and $\Pi$ be the set of probability measures compatible with $\lp$. For any event $A \in \mathscr{B}(\Omega)$, under the generalized Bayes rule,
\begin{itemize}
	\item $\mathcal{B}$ cannot induce sure loss in $A$,
	\item $\mathcal{B}$ cannot contract $A$.
\end{itemize}
\end{theorem}
The first part of Theorem \ref{thm:genbayes-coherence}, that the generalized Bayes rule avoids sure loss, is well-known in the literature and is the very reason that many authors such as \cite{walley1991statistical} and \cite{jaffray1992bayesian} consider it to be the sole choice as coherent updating rule, or the ``conditioning proper''. However, as we will show next, the generalized Bayes rule is also the most prone to dilation.

\subsection{Generalized Bayes rule produces a superset of probability measures}

\begin{lemma}[\textit{Generalized Bayes rule produces the widest intervals}]\label{lemma:gen_Bayes_widest}
For all $A, B \in \mathscr{B}\left(\Omega\right)$ such that the following quantities are defined, we have
\begin{equation}\label{eq:wider-geom}
\lpb\left(A\mid B\right)\le  \lpd\left(A\mid B\right) \le \upd\left(A\mid B\right)  \le\upb\left(A\mid B\right)
\end{equation}
and
\begin{equation}\label{eq:wider-genbayes}
\lpb\left(A\mid B\right) \le  \lpg\left(A\mid B\right) \le \upg\left(A\mid B\right)  \le \upb\left(A\mid B\right).
\end{equation}
\end{lemma}

That is, the conditional probability intervals resulting from Dempster's rule and the Geometric rule are always contained within those of the generalized Bayes rule. The fact that Dempster's rule produces shorter posterior intervals than that of the generalized Bayesian rule was discussed in \cite{dempster1967upper} and \cite{kyburg1987bayesian}. Here we supply a straightforward proof that applies to both sharper rules.

\begin{proof}
For Dempster's rule, the conditional plausibility function satisfies
\begin{equation*}
\upd \left(A\mid B\right) = \frac{\sup_{P\in\Pi}P\left(A\cap B\right)}{\sup_{P\in\Pi}P\left(B\right)} \le  \sup_{P\in\Pi} \frac{P\left(A\cap B\right)}{P\left(B\right)}
  = \upb \left( A\mid B\right)
\end{equation*}
and by conjugacy, also $\lpd \left(A\mid B\right) \ge \lpb \left(A\mid B\right) $. Similarly for the Geometric rule, the conditional lower probability function satisfies
\begin{equation*}
\lpg \left(A\mid B\right) = \frac{\inf_{P\in\Pi}P\left(A\cap B\right)}{\inf_{P\in\Pi}P\left(B\right)} \ge  \inf_{P\in\Pi} \frac{P\left(A\cap B\right)}{P\left(B\right)}
  = \lpb \left( A\mid B\right)
\end{equation*}
and by conjugacy, also $\upg \left(A\mid B\right) \le \upb \left(A\mid B\right) $.
\end{proof}

\begin{theorem}[\textit{Generalized Bayes rule}]\label{thm:gen_Bayes_superset}
Let $B \in \mathscr{B}\left(\Omega\right)$ be such that $\lp(B) > 0$. Denote sets of posterior probability measures $\Pi_{\mathfrak{B}}=\left\{ P:P\ge\lpb \left(\cdot\mid B\right)\right\}$, $\Pi_{\mathfrak{D}}=\left\{ P:P\ge\lpd \left(\cdot\mid B\right) \right\}$ and $\Pi_{\mathfrak{G}}=\left\{ P:P\ge\lpg\left(\cdot\mid B\right)\right\}$. Then,
\begin{equation}
	\Pi_{\mathfrak{G}}\subseteq\Pi_{\mathfrak{B}} \qquad \text{and}  \qquad 	\Pi_{\mathfrak{D}}\subseteq\Pi_{\mathfrak{B}}. 
\end{equation}
\end{theorem}

Theorem \ref{thm:gen_Bayes_superset} is a direct consequence of Lemma \ref{lemma:gen_Bayes_widest}, noting that $\Pi_{\mathfrak{G}}$, $\Pi_{\mathfrak{B}}$ and $\Pi_{\mathfrak{D}}$ are all convex and closed. Two more consequences of Lemma \ref{lemma:gen_Bayes_widest} are stated below; in particular, Examples~\ref{example:prisoner} and \ref{example:election} are respective embodiments of the two corollaries below.

\begin{corollary}
If $\mathcal{B}$ incurs sure loss in $A$ under Dempster's rule and sure gain under the Geometric rule, or vice versa, then $\mathcal{B}$ strictly dilates $A$ under generalized Bayesian rule.
\end{corollary}

\begin{corollary}\label{coro:gen_Bayes_necessary}
	If $\mathcal{B}$ (strictly) dilates $A$ under either Dempster's rule or the Geometric rule, then $\mathcal{B}$ (strictly) dilates $A$ under generalized Bayesian rule.
\end{corollary}

Theorem 2.1 of \cite{seidenfeld1993dilation} stated that, if dilation occurs with the generalized Bayesian rule, the associated set of probabilities $\Pi$ has a non-empty intersection with that of the independence plane between $A$ and $B$. Thus following Corollary \ref{coro:gen_Bayes_necessary} we have

\begin{corollary}
If $\mathcal{B} = \{B, B^c\}$ dilates $A$ under either Dempster's rule or the Geometric rule, then there exists $ P^{*}\ge\lp$ such that
\begin{equation}
P^{*}\left(A\cap B\right)=P^{*}\left(A\right)P^{*}\left(B\right).
\end{equation}
\end{corollary}
That is, dilation of an event by a binary partition under either Dempster's or the Geometric rules is a necessary condition for the posited set of probabilities to contemplate event independence, since posterior intervals under both rules are contained within the generalized Bayes posterior interval.

\subsection{Generalized Bayes rule and Geometric rule cannot sharpen vacuous prior intervals}

\begin{theorem}[\textit{Sharpening of vacuous intervals}]\label{thm:vacuous-interval}
	Let $\lp$ be such that $\lp \left(A\right) = 0$, $\up \left(A\right) = 1$. For any $B \in \mathscr{B}\left(\Omega\right)$ such that $\lp\left(B\right) > 0$, we have
	\begin{equation} \label{eq:vacuous-geometric}
		 \lpg \left(A\mid B\right)  = 0 \, ,\qquad  \upg \left(A\mid B\right) = 1
	\end{equation}
	and
	\begin{equation} \label{eq:vacuous-genbayes}
		 \lpb \left(A\mid B\right)  = 0 \, ,\qquad  \upb \left(A\mid B\right) = 1.
	\end{equation}
\end{theorem}

\begin{proof}
Notice that if $\lp \left(A\right) = 0$ and  $\up \left(A\right) = 1$, then $\lp\left(A\cap B\right) = \lp\left(A^{c} \cap B\right) =0$ for any $B$. Therefore, by   (\ref{eq:lpg}) we have
\[
\lpg \left(A\mid B\right) = \lp\left(A\cap B\right)/\lp\left(B\right) = 0
\]
and $\upg\left(A\mid B\right)=1-\lpg \left(A^{c}\mid B\right)=1$ provided that the denominator is greater than zero. Furthermore, by (\ref{eq:wider-geom}) we have $\lpb\left(A\mid B\right) \le \lpg \left(A\mid B\right) = 0$ and $\upb\left(A\mid B\right) \ge \upg \left(A\mid B\right) = 1$.
\end{proof}

The liberty to express partially lacking, and vacuous, prior knowledge is a prized advantage of imprecise probability over their precise, or full Bayesian, counterparts. Theorem \ref{thm:vacuous-interval} shows that both the generalized Bayes rule and Geometric rule are incapable of revising a vacuous prior interval to something informative for any possible outcome in the event space, whereas Dempster's rule is capable of such revision, Example~\ref{example:coin} being an instance. This again highlights the non-negligible influence imposed by the rule itself, as well as the difficulty to deliver all desirable properties in one single rule. Avoiding sure loss and being able to update from complete ignorance both seem to be rather basic requirements, but together they are sufficient to eliminate all three rules studied here. The following result perhaps is even more disturbing, because it says that in the world of imprecise probabilities, not only must we live with imperfections, but also accept intrinsic contradictions.

\subsection{The counteractions of Dempster's rule and Geometric rule} \label{subsec:result-counteract}

\begin{theorem}
If $\mathcal{B}=\left\{ B,B^{c}\right\}$ dilates $A$ under the Geometric rule, then it must contract $A$ under Dempster's rule. Similarly, if $\mathcal{B}$ dilates $A$ under Dempster's rule, then it must contract $A$ under the Geometric rule.  In both cases, the contraction is strict if the corresponding dilation is strict. 
\end{theorem}

\begin{proof}
If $\mathcal{B}$ strictly dilates $A$ under the Geometric rule, then for either $Z\in\mathcal{B}$
\begin{eqnarray}
\lpg\left(A\mid Z \right) &=& \frac{\lp \left(A\cap Z\right)}{\lp \left(Z\right)} < \lp\left(A\right) \label{eq:geom-lower-dilate}, \\
\upg\left(A\mid Z\right) &=& \frac{\up \left(A\cup Z^{c}\right)-\up \left(Z^{c}\right)}{\lp \left(Z\right)} > \up \left(A\right). \label{eq:geom-upper-dilate}
\end{eqnarray}
It follows then  
\begin{eqnarray*}
	\frac{\upd\left(A\mid B\right)}{\up\left(A\right)}
 & = & \frac{\up\left(A\cap B\right)}{\up \left(A\right)\cdot\up\left(B\right)} = \frac{\up\left(A\cap B\right)}{\up \left(A\right)\cdot\left(1-\lp\left(B^{c}\right)\right)}\\
 & < & \frac{\up\left(A\cap B\right)}{\up\left(A\right)\cdot\left[1-\left(\up\left(A\cup B\right)-\up\left(B\right)\right)/\up\left(A\right)\right]}\\
 & = & \frac{\up\left(A\cap B\right)}{\up\left(A\right)+\up\left(B\right)-\up\left(A\cup B\right)} \le 1,
\end{eqnarray*}
where the first inequality follows from (\ref{eq:geom-upper-dilate}) with $Z=B^c$, and the second inequality is based on the $2$-alternating nature of $\up$. (The $2$-alternating nature was also implicitly used in the first inequality to ensure $\up\left(A\cup B\right)-\up\left(B\right) < \up\left(A\right)$, hence the positivity of the denominator after replacing $\lp(B^c)$ with an upper bound.) In a similar vein,
\begin{eqnarray*}
\frac{\lpd \left(A\mid B\right)}{\lp \left(A\right)} & = & \frac{\up\left(B\right)-\up \left(A^{c}\cap B\right)}{\up \left(B\right) \cdot \lp \left(A\right)} = \frac{\lp \left(A\cup B^{c}\right)-\lp \left(B^{c}\right)}{\up \left(B\right) \cdot \lp \left(A\right)}\\
 & \ge & \frac{\lp \left(A\right)-\lp \left(A\cap B^{c}\right)}{\left(1-\lp \left(B^{c}\right)\right)\cdot\lp \left(A\right)} \\
 & = &  \frac{\lp \left(A\right)-\lp \left(A\cap B^{c}\right)}{\lp \left(A\right)-\lp \left(B^{c}\right)\cdot\lp \left(A\right)}>1,
\end{eqnarray*}
where the first inequality uses the $2$-monotone nature of $\lp$ and the second inequality is based on (\ref{eq:geom-lower-dilate}) with $Z=B$. Thus we have $\upd\left(A\mid B\right) < \up\left(A\right)$ and $\lpd\left(A\mid B\right) > \up\left(A\right)$, and clearly both inequalities still hold when we replace $B$ by $B^c$ because 
(\ref{eq:geom-lower-dilate})-(\ref{eq:geom-upper-dilate}) hold for both $Z=B$ and $Z=B^c$,  hence $\mathcal{B}$ strictly contracts $A$ under Dempster's rule. If $\mathcal{B}$ dilates $A$ under the Geometric rule but not strictly, the inequality in either  (\ref{eq:geom-lower-dilate}) or (\ref{eq:geom-upper-dilate}), but not both, may hold with equality, hence $\mathcal{B}$ contracts $A$ under Dempster's rule but not strictly. This completes the proof for the first half of the statement. 

For the second half, when $\mathcal{B}$ strictly dilates $A$ under Dempster's rule, we have for any $Z\in\mathcal{B}$,
\begin{eqnarray*}
\lpd (A\mid Z) & = & \frac{\up \left(A\cap Z\right)}{\up \left(Z\right)}  > \up \left(A\right), \\
\upd (A\mid Z)  & = & \frac{\lp \left(A\cup Z^{c}\right)-\lp \left(Z^{c}\right)}{\up \left(Z\right)} < \lp \left(A\right).
\end{eqnarray*}
Noting both inequalities hold for $Z$ and $Z^c$, we have
\begin{equation*}
	1 > \frac{\lp \left(A\cup Z \right)-\lp \left(Z\right)}{\lp \left(A\right)\cdot\up \left(Z^{c}\right)} 
 \ge 	\frac{\lp \left(A\right)-\lp \left(A\cap Z\right)}{\lp \left(A\right)-\lp \left(A\right)\cdot\lp \left(Z\right)}.
\end{equation*}
Hence 
$ \lp \left(A\right) < \lp \left(A\cap Z\right) / \lp \left(Z\right) = \lpg (A\mid Z)$.
On the other hand,
\begin{equation*}
1  < \frac{\up \left(A\cap Z^{c}\right)}{\up \left(A\right)\cdot\up \left(Z^{c}\right)} 
   \le  \frac{\up \left(A\right)-\left(\up \left(A\cup Z^{c}\right)-\up \left(Z^{c}\right)\right)}{\up \left(A\right)-\up \left(A\right)\cdot\lp \left(Z\right)}. 
\end{equation*}
Hence $\up \left(A\right)  >  \left( \up \left(A\cup Z^{c}\right)-\up \left(Z^{c}\right) \right) / \lp \left(Z\right) =\upg \left(A\mid Z\right)$.
The same argument applies that if $\mathcal{B}$ dilates $A$ under Dempster's rule but not strictly, it contracts $A$ under the Geometric rule but not strictly. This completes the proof for the second half of the statement. 
\end{proof}

\subsection{Visualizing Relationships and Complications}\label{subsec:poll}

\begin{example}{5}[\textit{Pre-election poll}]
\label{example:election}	

Suppose that we intend to study the voter intention prior to the 2016 US election. For simplicity, assume there are only two parties, represented respectively by Clinton and Trump, with one to be elected. The pre-election poll consists of two questions:
\begin{enumerate}
\item Do you intend to vote for Trump or Clinton?
\item Do you identify more as a Republican or a Democrat?
\end{enumerate}

Among all surveyed individuals, some answered both questions, some only one, and the rest did not respond. Let $Q_1 = \{\text{Trump}, \text{Clinton}\}$ denote votes for Trump and Clinton, respectively, and $Q_2 = \{\text{Republican}, \text{Democrat}\}$ denote identification with the Republican and Democratic parties, respectively. If all the percentages of response patterns are fully known, this model can be represented as a belief function. Assume the mass function $m\left(\cdot\right)$ reflecting the coarsened sampling distribution for these set-valued observations appears as Table~\ref{table:poll} (of course, the numbers are for illustrations only):

\begin{table}[H]
\caption{\label{table:poll} Hypothetical data from a voter poll consisting of two questions}
\small
\def\arraystretch{1.5}
\begin{tabular}{|c|c|c|c|c|c|c|c|c|c|}
\hline 
$Q_{1}$ & C & T & C & T & C & T & (n/a) & (n/a) & (n/a)\tabularnewline
\hline 
$Q_{2}$ & Dem & Dem & Rep & Rep & (n/a) & (n/a) & Dem & Rep & (n/a)\tabularnewline
\hline
\hline 
$m(\cdot)$ & \multicolumn{8}{c|}{$0.1 - \epsilon$} & $0.2 + 8\epsilon$ \tabularnewline
\hline 
\end{tabular}
\normalsize
\end{table}

A ``tuning parameter'' $\epsilon\in\left[-0.025,0.1\right]$ is installed to create a family of mass function specifications in order to investigate the differential behavior among updating rules as a function of the coarseness of the data. The smaller the $\epsilon$, the more the mass function concentrates on the precise observations (more survey questions answered). The larger the $\epsilon$, the closer the random set approaches the vacuous belief function. As a function of $\epsilon$, the prior lower and upper probabilities for Clinton are
\begin{equation*}
	\lp \left( \text{C} \right) = 0.3 - 3\epsilon \, , \qquad \up \left( \text{C} \right) = 0.7 + 3\epsilon.
\end{equation*}
The prior lower and upper probabilities for Trump, as well as for identification of either parties are numerically identical to the above, since the setup is fully symmetric with respect to both voting intention and partisanship. For example, when $\epsilon=0$, the table above shows that altogether 40\% of the respondents diligently answered both questions, 20\% only identified prior partisanship, 20\% only expressed current voting intentions, and another 20\% did not respond at all. Thus, $m\left(\cdot\right)$ determines a pair of belief and plausibility functions which bounds  the vote share for both Clinton and Trump to be within $30\%$ and $70\%$. 

How will information on partisanship affect the knowledge on voting intention? According to the three updating rules, the lower and upper probabilities for Clinton conditional on either values of partisanship $Q_2$, are as follows:
\begin{eqnarray*}
\lpb \left( \text{C} \mid Q_2 \right) = \frac{0.1-\epsilon}{0.6+4\epsilon}, & \qquad & \upb \left( \text{C} \mid Q_2 \right) = \frac{0.5+5\epsilon}{0.6+4\epsilon},  \\
\lpd \left( \text{C} \mid Q_2 \right) = \frac{0.2-2\epsilon}{0.7+3\epsilon},
  & \qquad & \upd \left( \text{C} \mid Q_2 \right) = \frac{0.5+5\epsilon}{0.7+3\epsilon}, \\
\lpg \left( \text{C} \mid Q_2 \right) = \frac{1}{3}, &  \qquad & \upg \left( \text{C} \mid Q_2 \right) = \frac{2}{3}.
\end{eqnarray*}
\end{example}

\begin{figure}
\noindent \begin{centering}
\includegraphics[width=\textwidth]{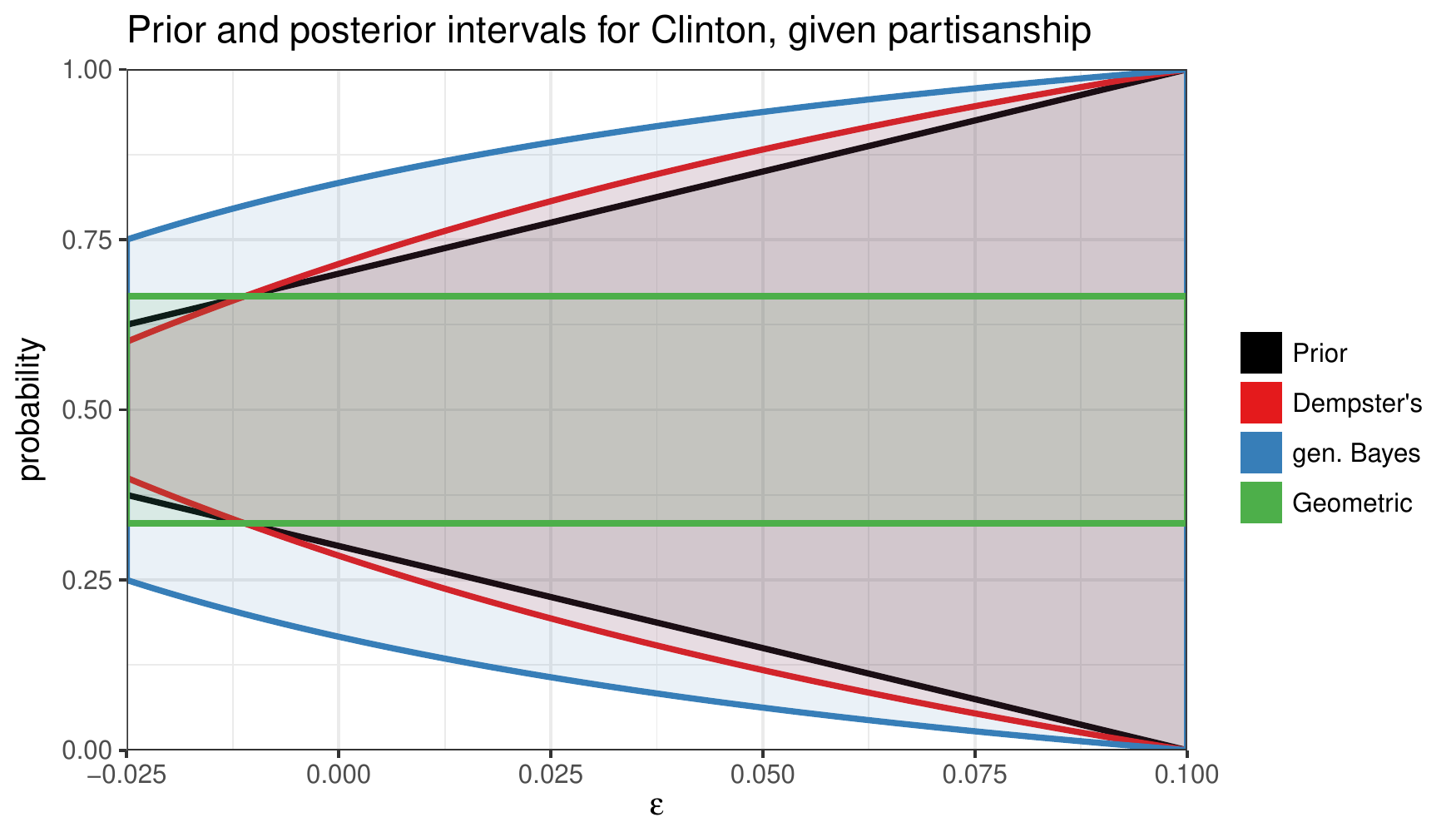}
\par\end{centering}
\caption{\label{fig:prob_A_epsilon} Prior probability interval for Clinton's voter support (black) and posterior probability intervals given reported partisanship according to the three updating rules (blue: generalized Bayes, red: Dempster's, green: Geometric). Due to full symmetry of the setup, {contraction} happens under an updating rule whenever the corresponding posterior interval depicted is contained within the prior interval; vice versa for {dilation}.}
\end{figure}

See Figure \ref{fig:prob_A_epsilon} for the above quantities as functions of $\epsilon$. We observe that
\begin{itemize}
	\item Under {\it the generalized Bayes rule}, knowledge about partisanship strictly dilates voting intention for either candidate for all $\epsilon < 0.1$. That is to say, learning the prior partisanship of an individual dilates our inference of her current voting intention, and vice versa, and this is true no matter which party or candidate is said to be favored;
	\item Under {\it Dempster's rule}, partisanship strictly dilates voting intention for either candidate for $-0.011 < \epsilon < 0.1$, and strictly contracts both for $-0.025 < \epsilon < -0.011$;
	\item Under {\it the Geometric rule}, partisanship strictly dilates voting intention for either candidate for $-0.025 < \epsilon < -0.011$, and strictly contracts both for $-0.011 < \epsilon <0.1$. Moreover, the absolute value of the lower and upper posterior probability remained constant regardless of the value of $\epsilon$.
	\end{itemize}

Furthermore, we observe some of the phenomena discussed previously in this section. For example, the extent of dilation exhibited by the generalized Bayes rule is to a strictly larger extent than that of both Dempster's rule and the Geometric rule, if either of them does dilate. The dilation-contraction status of Dempster's rule and the Geometric rule are in full opposition to each other, switching precisely at $\epsilon = -0.011$.

\section{Food For Thoughts}\label{sec:discussion}

\subsection{Assumption incommensurability and conditioning protocol}\label{subsec:verdict}

As revealed in Section~\ref{subsec:unmeasurable}, each imprecise probability updating rule is constantly faced with the problem that the conditioning information may not be measurable with respect to the very imprecise probability it is trying to update. As a consequence, they each effectively build within themselves a  mechanism for imposing mathematical restrictions generated by a given event $B$. 
This is why, as far as we can see, the situation in the world of imprecise probability is more confusing and clearer at the same time. It is more confusing because the notations $\lps$ and $\ups$ carry meanings contingent upon the $\bullet$-rule we choose. Yet, different rules are built upon different mechanisms for imposing the mathematical restriction specified by an event partition $\mathcal{B}$, in a much similar vein to the sampling and missing data mechanisms mentioned previously, potentially supplying a variety of options suitable for different situations that users may choose from, as long as they are well-informed of the implied assumptions of each rule. In this sense the situation is clearer, because the imprecise nature should compel the users to be explicit about the imposed mechanisms in order to proceed. Below we illustrate this point.

%
%


\begin{example}{\ref{example:boxer} cont.}[\textit{The boxer, the wrestler, and the God's coin}]
Recall the boxer and wrestler example in which there exists  {\it a priori}, a fair coin and vacuous knowledge of the two fighters. Our analysis in Section~\ref{sec:updates} showed that upon knowing $X=Y$, Dempster's rule will judge the posterior probability of boxer's win to be precisely half, whereas generalized Bayes rule will remain that the chance is anywhere within $[0, 1]$. We realize that the witness who relayed the message $X=Y$ could have meant it in (at least) two different ways:
\begin{enumerate}
	\item that he happened to see both the coin flip and the match between the two fighters, and the results of the two events were identical;
	\item that he, somehow miraculously, knew that the coin toss \emph{decides} the outcome of the match, as if the coin is God's pseudorandom number generator.
\end{enumerate}

If the first meaning is taken, as most of us naturally do, it seems that the generalized Bayes answer makes sense. After all, since we do not know the relationship between two co-observed phenomenon, the worst case scenario would be to admit all possibilities, including the most extreme forms of dependence, when deriving the probability interval. 

However, if the head of the coin dictates the triumph of the boxer, and the former event is known precisely as a toss-up, it makes sense to think of the match as a true toss-up as well. In this case, it is rightful to call for a transferral of the  {\it a priori}  precise probability of $X$ onto the  {\it a priori}  vacuous $Y$. The same logic would apply had we been told $X\neq Y$, in the sense that the head of the coin dictates the triumph of the wrestler. In both cases, the update is akin to adding another piece of structural knowledge to the model itself. 
\end{example}

This example reflects well a point made by \cite{shafer1985conditional}. In order for probabilistic conditioning to be properly interpreted, it is crucial to have a ``protocol'' specifying what information {\it can} be learned, in addition to learning the actual information itself. Updating in absence of a protocol, or worse, under an unacknowledged, implicit protocol can produce dangerous complications to the interpretation of the output inference. Dilation and sure loss, phenomena exclusive to imprecise probability, are striking instances that demonstrate such danger. The difference among the three imprecise model updating rules discussed in this paper is precisely the way the same incoming message might be interpreted. Each conditioning rule effectively creates a world of alternative possible observations, hence a protocol is de-facto in place, only hidden behind these explicit-looking rules. 

When performing updates in the boxer and wrestler's case, the distinction between conditioning protocols underlying the solutions we have offered so far is one between {\it factual} versus {\it incidental} knowledge spaces.  Knowing $X = Y$ is a possible outcome and by chance observing it constitutes incidental knowledge. Knowing that $X = Y$ is the factual state of the nature is knowledge of a fundamentally different type, one that is much more restrictive and powerful at the same time: in other words, $X\not=Y$ cannot, could not, and will not happen. Unlike their incidental counterparts, claiming either $X = Y$ or $X \neq Y$ as factual necessarily makes them incommensurable with one another, even over sampling repetitions. That is to say, if either $X = Y$ or $X \neq Y$ are to be hard-coded into the model, they will each result in a model distinct from the other in a way that their respective posterior judgments about the same event, say $Y = 1$, are not meant to enter the same law of total probability calculation. If we are willing to admit either $X = Y$ or $X \neq Y$ as factual evidence to condition on, they can no longer be regarded as a partition of the full space like they did back in Section~\ref{subsec:dilation}; the model must also anticipate to deal with a whole range of other possible relationships between $X$ and $Y$ that are non-deterministic, as part of the conditioning protocol in \citeauthor{shafer1985conditional}'s sense.

The distinction drawn here between the judgments of factual knowledge versus incidental knowledge are referred to as {\it revision} versus {\it focusing} in the imprecise probability literature; see \cite{smets1991updating} for a discussion on the matter and its reflected ideologies of update rules. 
Whether a rule is applicable to a particular imprecise model would consequently depend on the type of knowledge we have and what questions we want to answer.


Within a precise modeling framework, the knowledge type for conditioning can typically be coded into the conditioning event itself, which might be on an enhanced probabilistic space but without increasing  the resolution of the original (marginal) model because it is already at the highest possible resolution.  Hence one universal updating rule is sufficient. Under an imprecise model, such a resolution-preserving  encoding may not be possible because of the low resolution nature of the original model. Various rules then were/are invented to carry out the update as a qualitative rescue for the model's inability to quantify the knowledge types within its original resolution. This makes the judgment of knowledge types particularly pronounced, which serves well as a loud reminder of the central and precise nature of conditioning in statistical learning. But this also increases vulnerability and confusion when we do not explicate the applicability and subtitles of each updating rule, leading to seemingly paradoxical phenomena studied in this paper.

\subsection{Imprecise probability, precise decisions}\label{subsec:decision}

Seeing a myriad of sensible and non-sensible answers produced by the updating rules of imprecise models, one may wonder if anything certain, or close to certain, can be inferred from these models at all without stirring up a controversy. To this end, we discuss a final twist to the three prisoners' story.  

\begin{example}{\ref{example:prisoner} cont.}[\textit{Three prisoners' Monty Hall}]
	Having heard from the guard that $B$ will not receive parole, prisoner $A$ is presented with an option to switch his identity with prisoner $C$: that is, the next morning $A$ will be met with the fate of $C$ (and $C$ that of $A$), both having been decided unbeknownst to them. Is this a good idea for $A$?
\end{example}
The answer is unequivocally yes. The above is a recast of the Monty Hall problem in which you, the contestant standing in front of a randomly chosen door (prisoner $A$), have just been shown a door with a goat behind it (``$B$ will be executed''), and are contemplating a switch to the other unopened door (the identity of prisoner $C$) for a better chance of winning the new car (parole). By the calculations in (\ref{eq:prisoner-a-given-b}), we know that under the generalized Bayes rule
\begin{equation*}
\upb\left(A\text{ lives}\mid\text{guard says }B\right) = \lpb\left(C\text{ lives}\mid\text{guard says }B\right) 
\end{equation*}
suggesting that a switch will under no circumstance hurt the chance of $A$'s survival, because without switching, $A$'s best chance of surviving does not exceed $C$'s worst chance of living. Moreover, as the most conservative rule of all, the (almost) separation of the two posterior probability intervals from the generalized Bayes rule guarantees the same for the other updating rules as well. As far as $A$ is concerned, the action of identity switching can be recommended without reservation, regardless of the choice of rule among the three discussed. This possibility is due to the (very) low resolution nature of the action space, often binary (e.g., switching or not), allowing different high-resolution probabilistic statements to admit the same low resolution classification in the action space.

\subsection{Beyond associations: extended types of information contribution in imprecise models}

As discussed at the beginning of Section~\ref{sec:intro}, in precise probability models, association is the fundamental means through which observed information contribute to the model, as it characterizes how probabilistic information should change with one after the other has been learned, and vice versa. The sign of the association gives the sense of direction, such as seen from the coefficients in regression models. The magnitude of the association implies an order of priority, such as in large scale genome-wide association studies and elsewhere where correlation coefficients are used as test statistics. Plentiful association is the indication of signal strength, potential discovery, and the prospect of a causal relationship. The absence of association, on the other hand, is just as desirable when used to justify independence assumptions, creating a blanket of simplicity on which small-world models can be built and trusted. The three types of associations (positive, negative, and independence) are in one-to-one correspondence with the three possible directions of change as the probability of an event updates from the prior to the posterior according to the Bayes rule. In precise probabilities, these three types of associations exhaust all possibilities of information contribution from one event to another.

Imprecise probabilities complicate the landscape of information contribution fundamentally because the probabilistic description assigned to each event is no longer singular. For coherent upper and lower probabilities, this description is now a closed interval $[\lp(A), \up(A)]$ of possibly non-negligible width. As a consequence, characterizations of the direction of change from prior to posterior becomes a complicated topic. Generalized versions of the classical notion of association and independence are yet to be defined for sets of probabilities. Furthermore, novel phenomena like dilation, contraction and sure loss as explored in this paper are hinting at extended types of information contribution waiting to be harnessed. As imprecise probability updating is contingent upon the choice of rules, can we use some of these rules, capable of an array of novel behavior such as Dempster's rule, to characterize a generalized set of definitions of information contribution in imprecise probabilities?

\section*{Acknowledgements}

We thank Arthur Dempster, Keli Liu, Glenn Shafer and Teddy Seidenfeld for very helpful discussions and comments, and Steve Finch for proofreading.

\bibliographystyle{apacite}
\bibliography{main}

\end{document}